\documentclass[oneside]{amsart}
\usepackage{amssymb, amsmath, amsthm}
\usepackage{graphicx}
\usepackage[all]{xy}
\xyoption{arc}
\CompileMatrices

\newcommand{\comment}[1]{}
\newcommand{\smallcaps}[1]{\textrm{\textsc{#1}}}
\newcommand{\brac}[1]{\left(#1\right)}
\newcommand{\bbrac}[1]{\bigl(#1\bigr)} 
\newcommand{\Brac}[1]{\Bigl(#1\Bigr)}

\newcommand{\sbrac}[1]{\left[#1\right]}

\newcommand{\cbrac}[1]{{\left\{#1\right\}}}
\newcommand{\cbbrac}[1]{{\bigl\{#1\bigr\}}}
\newcommand{\mbrac}[1]{\mid\!\!#1\!\!\mid}

\newcommand{\NN}{\mathbb{N}}

\newcommand{\ZZ}{\mathbb{Z}}

\newcommand{\ttt}{[\![t]\!]}

\newcommand{\sym}{\mathfrak{S}}
\newcommand{\symn}{\sym_n}
\newcommand{\symG}{\sym_\bdG}
\newcommand{\dt}{\bullet}
\newcommand{\clc}{,\ldots,}
\newcommand{\tlt}{\times\ldots\times}
\newcommand{\ala}{\ast\ldots\ast}

\newcommand{\xytree}[1]{\vcenter{\xymatrix@C=1pc@R=1pc{#1}}}
\newcommand{\smtree}[1]{\vcenter{\xymatrix@C=0.6pc@R=0.6pc{#1}}}

\newcommand{\ed}[1]{\overrightarrow{#1}}
\newcommand{\Forn}{{\mathcal{F}_n}}
\newcommand{\gf}{\gamma_{\Forn}}
\newcommand{\Gf}{{\Gamma_{\!\Forn}}}
\newcommand{\FR}{\smallcaps{FR}}
\newcommand{\SFR}{\Sigma\FR}
\newcommand{\Wh}{\smallcaps{Wh}}
\newcommand{\SWh}{\Sigma\Wh}
\newcommand{\SAut}{\Sigma\smallcaps{Aut}}
\newcommand{\PAut}{\smallcaps{PAut}}
\newcommand{\Pf}{P_f}
\newcommand{\Gam}[1]{(\Gamma_{#1},\gamma_{#1})}
\newcommand{\Gaml}[1]{(\Gamma_{#1},\gamma_{#1},l)}
\newcommand{\bG}{\mathbf{\Gamma}}
\newcommand{\alphaij}{\alpha_i^{g_j}}
\newcommand{\bdG}{\mathbf{G}}
\newcommand{\bdY}{\mathbf{Y}}
\newcommand{\bdH}{\mathbf{H}}
\newcommand{\YGam}[1]{\bdY\Gam{#1}}
\newcommand{\YpGam}[1]{\bdY'\Gam{#1}}
\newcommand{\GGam}[1]{\bdG\Gam{#1}}

\newcommand{\YGf}{\mathbf{Y}\Gf}
\newcommand{\rC}{\widehat{C}_\ast}
\newcommand{\rcC}{\widehat{C}^\ast}

\newcommand{\FZ}{\smallcaps{Forests}_Z}
\newcommand{\out}[1]{{\smallcaps{out}(#1)}}

\newcommand{\cat}[1]{\mathcal{#1}}
\newcommand{\edij}{\ed{ij}}
\newcommand{\edik}{\ed{ik}}

\newcommand{\cUi}{{\cbrac{U_i}}}
\newcommand{\cVj}{{\cbrac{V_j}}}
\newcommand{\cWk}{{\cbrac{W_k}}}
\newcommand{\bAut}{\textbf{Aut}}
\newcommand{\bInn}{\textbf{Inn}}
\newcommand{\MY}{\mathcal{M}\bdY}

\newcommand{\HH}{\mathcal{H}}
\newcommand{\hh}{\mathfrak{H}}
\newcommand{\UU}{\mathfrak{U}}
\newcommand{\Xs}{X^{\text{simp}}}
\newcommand{\XsG}[1]{X^{\text{simp}}(#1)\uparrow^G}

\DeclareMathOperator{\Aut}{Aut}

\DeclareMathOperator{\im}{Im}
\DeclareMathOperator{\Inn}{Inn}

\DeclareMathOperator{\Tor}{Tor}
\DeclareMathOperator{\colim}{colim}

\theoremstyle{plain}
\newtheorem{theorem}{Theorem}[section]
\newtheorem{lemma}[theorem]{Lemma}
\newtheorem{proposition}[theorem]{Proposition}
\newtheorem{corollary}[theorem]{Corollary}
\newtheorem{mainthm}{Theorem}

\newtheorem{question}{Question}

\theoremstyle{definition}
\newtheorem{definition}{Definition}
\newtheorem{example}{Example}[section]

\theoremstyle{remark}
\newtheorem{remark}{Remark}[section]

\numberwithin{equation}{subsection}

\title[Diagonal complexes and $H_\ast(\Aut(G_1\ala G_n), R)$]{Diagonal complexes and the integral homology of the automorphism group of a free product}
\author{James Griffin}

\begin{document}

\begin{abstract}
The main goal of this paper is a calculation of the integral\linebreak (co)homology of the group of symmetric automorphisms of a free product.
We proceed by giving a geometric interpretation of symmetric automorphisms via a moduli space of certain diagrams, which we name cactus products.
To describe this moduli space a theory of diagonal complexes is introduced.
This offers a generalisation of the theory of right-angled Artin groups in that each diagonal complex defines what we call a diagonal right-angled Artin group (DRAAG).
\end{abstract}

\maketitle

\section*{Introduction}
This paper studies the group $\SAut(G)$ of symmetric automorphisms of a free product of groups $G=G_1\ala G_n$.  
At the heart of the paper is a certain moduli space which offers an intuitive and geometric interpretation of the symmetric automorphisms.
The combinatorial backbone is provided by the poset of `forest posets'; it will parametrize the moduli spaces, present the automorphism groups and index their homology, the structure of which is our main result.
The combinatorial arguments consist of only short lemmas concerning this poset.

In many cases, for example if each $G_i$ is finite, every automorphism is symmetric and so our results apply to the whole automorphism group.
However symmetric automorphisms are also of independent interest.
For instance the group $\SAut(\ZZ^{\ast n})$ is both a proper subgroup of the automorphism group of a free group and also the fundamental group of the configuration space of $n$ unknotted, unlinked loops embedded into 3-space, see for instance~\cite{Brendle2010}.
Under this latter realisation it has been called the \emph{loop braid group} and has found applications in the field of mathematical physics~\cite{Baez2007}.
Future occurrences of symmetric automorphism groups in the study of configuration spaces are to be expected.

Let $\bdG=(G_1\clc G_n)$ be an $n$-tuple of groups and let $G=G_1\ala G_n$ be their free product.
Examples of automorphisms of this group are given by
\begin{itemize}
\item[(a)] a permutation of isomorphic factors, 
\item[(b)] an inner automorphism of an individual factor,
\item[(c)] more generally any automorphism of an individual factor and 
\item[(d)] a partial conjugation, where for distinct $i,j$ and some $g_j\in G_j$, elements $g\in G_i$ are taken to $g^{g_j}$, whilst factors $G_k$ are fixed for $k\neq i$.
\end{itemize}
Choosing different combinations of these automorphisms generates different groups.
\begin{itemize}
\item The group of automorphisms generated by (a), (b), (c) and (d) is called the \emph{symmetric automorphism group} $\SAut(G)$.
When each $G_i$ is freely indecomposable and is not isomorphic to $\ZZ$ then $\SAut(G)$ is the whole automorphism group.
\item The permutations (a) generate $\symG$, a subgroup of the symmetric group~$\symn$.
\item The Fouxe-Rabinovitch group $\FR(G)$ is the subgroup generated by automorphisms of type~(d).
\item The Whitehead automorphism group $\Wh(G)$ is the subgroup generated by automorphisms of types (b) and (d).
\item The pure automorphism group $\PAut(G)$ is the subgroup generated by automorphisms of types (c) and (d).
\end{itemize}
A presentation for the full automorphism group including the case with infinite cyclic factors was given by Fouxe-Rabinovitch~\cite{Fouxe1940}, see also~\cite{Gilbert1987} for a more concise set of defining relations.
It was shown that $\FR(F_n)=\Wh(F_n)$ was of cohomological dimension $n-1$ in~\cite{Collins1989} using a symmetric analogue of Outer space.
In~\cite{Jensen2006} the integral cohomology ring of $\FR(F_n)$ was computed, confirming a conjecture made by Brownstein and Lee~\cite{Brownstein1993}.
They used an action of the group $\Wh(F_n)/\Inn(F_n)$ on the McCullough-Miller complex introduced in~\cite{McCullough1996} to define a spectral sequence calculating the cohomology, which they then show to collapse.
Using this same complex attempts have recently been made to calculate the Euler characteristic~\cite{Jensen2007euler} and the cohomology over a field~\cite{Berkove2010} of $\Wh(G)$.
We discuss their results in Remark~\ref{rem_discrepancy}.

Our approach proceeds by describing a space $\MY$ built functorially out of an $n$-tuple of pointed spaces $\bdY=(Y_1\clc Y_n)$.
We offer two constructions of the space $\MY$, one is an intuitive description as the moduli space of certain cactus products and the action of its fundamental group on $G$ is easily visualised.
The other description is purely combinatorial and this involves the development of a theory generalising graph products of groups.
The graphs are replaced by objects called `diagonal complexes' and the corresponding product is referred to as a `diagonal complex product'.
The diagonal complex $\Gf$ for the space $\MY$ is given by the poset of `forest posets'.

Using the combinatorial description we are able to prove that the fundamental group $\pi_1(\MY)$ is isomorphic to the Fouxe-Rabinovitch group with $G_i=\pi_1(Y_i)$.
Recall that a pointed space $X$ is aspherical if the homotopy groups $\pi_j(X)$ disappear for $j>1$.
We prove our first main theorem by interpreting a based version of McCullough-Miller space described in~\cite{Chen2005} as a coset complex.
\begin{mainthm}
Let $\bdY=(Y_1\clc Y_n)$ be aspherical pointed spaces.
Then the space $\MY$ is also aspherical.
\end{mainthm}
Calculating the homology of diagonal complex products is quite simple.  Using this we are able to prove our second main theorem.
\begin{mainthm}
Let $\bdG$ be an $n$-tuple of groups and $G$ be the $n$-fold free product of these groups.
Then 
\begin{equation*}H_\ast(\FR(G),R) \cong H_\ast(G^{n-1},R),\end{equation*}
where $G^{n-1}$ is the $(n-1)$-fold direct product of $G$.
This also holds for the cohomology.
\end{mainthm}
Our final main theorem uses a symmetric group action on $\Gf$ and a careful analysis of the action of $\Aut(G_i)$ on $\MY$ to give a calculation of the integral homology of $\Aut(G)$.
\begin{mainthm}
Let $G=G_1\ala G_n$ where $G_i$ is neither freely decomposable nor infinite cyclic.  Then the homology of the automorphism group of $G$ is given by
\[
H_\ast\bigl(\Aut(G),R\bigr) \cong H_\ast\bigl(\bAut(\bdG)\!\rtimes\!\symG,R\bigr)\oplus\!\!\!
\bigoplus_{f\in\smallcaps{Forests}_Z}\!\!\! H_\ast\brac{\bAut(\bdG)\!\rtimes\!\Aut(f), \rC(f)},
\]
where $\FZ$ is the set of $Z$-coloured planted forests and $\rC(f)$ is given by
\[ \bigotimes_{i\in\sbrac{n}} \rC(Y_{l(i)},R)^{\otimes \out{i}},\]
where $Y_i$ is a classifying space for $G_i$ for each $i=1\clc n$.
\end{mainthm}
\noindent The definition of a $Z$-coloured planted forest and of the complexes above are given in Section~\ref{section_fullautos}.
The theorem also has an analogue for the integral cohomology.

The paper is ordered as follows.  The first section covers some preliminary material; it describes various subgroups of $\Aut(G)$, in particular giving a presentation of $\FR(G)$, and then discusses some topological considerations.

In the second section the moduli space of cactus products is introduced, we describe how its fundamental group acts intuitively on the free product and then discuss the connection with the Outer space construction when $G\cong F_n$, the free group on $n$ letters.

The third section describes the general theory of diagonal complexes, gives some simple examples and introduces the diagonal complex product.  Presentations for the diagonal complex product of groups are given and the homology of the diagonal complex product of pointed spaces is calculated.

For the fourth section we recap material from~\cite{Abels1993} whilst applying it to particular subgroups of diagonal complex products of groups.
This is the tool we need to calculate the homotopy type of the diagonal complexes which follow.

In the fifth section we study the diagonal complex of forest posets $\Gf$ and prove that the diagonal complex products associated to it give classifying spaces for the groups $\FR(G)$.

Finally in Section~\ref{section_HAutFP} the homology groups of $\FR(G)$, $\PAut(G)$ and $\SAut(G)$ are calculated.

\section*{Acknowledgements}
This work was carried out while the author was funded by the Richard Metheringham Scholarship, awarded by the Worshipful Company of Cutlers.
I am grateful to Vladimir Dotsenko for introducing me to this family of groups, to Aur\'{e}lien Djament and Ga\"{e}l Collinet for feedback on an earlier version of this paper and to Christopher Brookes for his recommendations and guidance.

\section{Preliminaries}
The groups to be studied in this paper are introduced, then we address some topological considerations.
\subsection{A diagram of subgroups}
Let $\bdG=\brac{G_1\clc G_n}$ be an $n$-tuple of groups and $G$ their free product.  We refer to the $G_i$ as \emph{the factors of $G$}.
To begin we describe a diagram of subgroups of $\SAut(G)$, this is recalling material from~\cite{McCullough1996}.

A number of groups were introduced in the introduction and they are generated by various combinations of the automorphisms:
\begin{itemize}
\item[(a)] permutations of isomorphic factors, 
\item[(b)] inner automorphisms of an individual factor,
\item[(c)] more generally, automorphisms of an individual factor and 
\item[(d)] partial conjugations
\begin{equation}\alphaij(g) = \begin{cases}
g & \text{ if $g\in G_k$ and $k\neq i$}\\
g^{g_j} & \text{ if $g\in G_i$.}
\end{cases}\end{equation}
for distinct $i,j$ and some $g_j\in G_j$.
\end{itemize}
It will be convenient to identify isomorphic factors $G_i\cong G_j$, in particular this gives a canonical choice of subgroup of factor permutations which we denote $\symG$.
The subgroups generated by inner automorphisms and by automorphisms of the individual factors are denoted $\bInn(\bdG)\cong\prod_i\Inn(G_i)$ and $\bAut(\bdG)\cong\prod_i\Aut(G_i)$ respectively.

Recall that $\FR(G)$ is the group generated by partial conjugations and that $\Wh(G)$ is giving by adding in the inner factor automorphisms.
By adding on the symmetric group generators $\symG$ we get symmetric versions $\SFR(G)$ and $\SWh(G)$.
These groups fit into the diagram
\begin{equation}\xymatrix{
\SFR(G)\ar[r] &\ar[r] \SWh(G) & \SAut(G)\\
\FR(G)\ar[r]\ar[u] & \Wh(G)\ar[r]\ar[u] & \PAut(G).\ar[u]
}\end{equation}
Every arrow is a normal embedding as are their composites.  
Factoring out by the normal subgroup $\FR(G)$ gives the corresponding diagram.
\begin{equation}\xymatrix{
\symG\ar[r] &\ar[r] \bInn(\bdG)\rtimes\symG & \bAut(\bdG)\rtimes\symG\\
I\ar[r]\ar[u] & \bInn(\bdG)\ar[r]\ar[u] & \bAut(\bdG).\ar[u]
}\end{equation}
Most importantly we have that 
\begin{equation}\label{eq_semidirect}
\SAut(G) \cong \FR(G) \rtimes \bAut(\bdG) \rtimes \symG.
\end{equation}
We have omitted the bracketing deliberately, either semi-direct product, $\rtimes$ may be evaluated first.

The group $\FR(G)$ is presented as follows.
\begin{proposition}\label{prop_FRpres}
Let $\bdG=(G_1\clc G_n)$ and let $G=G_1\ala G_n$.  Then $\FR(G)$ is generated by partial conjugations
\begin{equation}\alpha_i^{g_j}\end{equation}
for $i=1\clc n$, with $j\neq i$ and $g_j\in G_j$.  These are subject to defining relations
\begin{align}
\alpha^{h_j}_i\alphaij &= \alpha_i^{g_jh_j}, && \label{eq_FRrel1}\\ 
\alpha^{e}_i &= e, &&\label{eq_FRrel2}\\
\sbrac{\alphaij, \alpha_k^{h_l}}&=e &&\text{ for $i\neq k$ and}\label{eq_FRrel3}\\
\sbrac{\alphaij\alpha_k^{g_j}, \alpha_i^{g_k}} &= e && \text{ for distinct $i,j$ and $k$.}\label{eq_FRrel4}
\end{align}
\end{proposition}
\noindent This is proved in~\cite{Gilbert1987} using peak reduction.
Our relations~\eqref{eq_FRrel1} and~\eqref{eq_FRrel2} are equivalent to the relations (1J) and (5)(iii).
Relation~\eqref{eq_FRrel3} is equivalent to the relations (5)(i), (5)(ii) and (2J).
Relation~\eqref{eq_FRrel4} is equivalent to (8J). 

In~\eqref{eq_semidirect} the symmetric automorphism group was given as a semi-direct product of $\FR(G)$ and $\bAut(\bdG)\rtimes\symG$.
Let $\sigma\in\symG$, then in $\SAut(G)$ we have
\begin{equation}
\sigma\alpha_i^{g_j}\sigma^{-1} = \alpha_{\sigma(i)}^{g_{\sigma(j)}}.
\end{equation}
Let $\tau\in\Aut(G_k)$ then
\begin{equation}
\tau\alpha_i^{g_j}\tau^{-1} = \begin{cases} \alpha_i^{\tau(g_j)} &\text{ if $j=k$},\\
\alpha_i^{g_j} & \text{ otherwise.} \end{cases}
\end{equation}
These suffice to give the action of $\bAut(\bdG)\rtimes\symG$ on $\FR(G)$.

\subsection{Some topological considerations}\label{section_topologicalconsiderations}
When we refer to a space $Y$, we will take this to mean that $Y$ is a CW~complex.
Any topological space is weakly homotopy equivalent to a CW~complex so for our purposes this is no restriction.
The direct product $Y_1\times Y_2$ of spaces $Y_1$ and $Y_2$ is taken to have the CW~product topology.  This differs only slightly from the usual product topology; both products have the same weak homotopy type.
For a discussion of these matters and further references see the Appendix of~\cite{Hatcher2002}.

The reason for restricting to CW~complexes is that a CW~complex may be reconstructed from a covering of CW~subcomplexes in the following sense.  
Let $\mathcal{U}=\cbrac{X_i}_{i\in I}$ be a cover of a CW~complex $X$ by CW~subcomplexes $X_i\hookrightarrow X$.
We may form a diagram $X_{\mathcal{U}}$ consisting of 
\begin{itemize}
\item the intersections $\bigcap_{i\in A} X_i$ for non-empty finite subsets $A\subseteq I$, and
\item the inclusions $\bigcap_{i\in A}X_i\hookrightarrow \bigcap_{j\in B} X_j$ for $B\subseteq A$.
\end{itemize}
The amalgamation $\text{am}(X_{\mathcal{U}})$ is the colimit of this diagram and has a natural map to $X$.
\begin{lemma}\label{lemma_amalgamations}
The natural map $\text{am}(X_{\mathcal{U}})\rightarrow X$ is an isomorphism of CW~complexes.
\end{lemma}
This follows from the definition of a CW~complex.
The map is always an isomorphism of sets but for topological spaces in general the topologies may differ, in fact even the weak homotopy types may be different.

A pointed space $(Y,\ast)$ consists of a space $Y$ with a chosen point $\ast\in Y$.  We will assume that each pointed space is path connected.

\section{The moduli space of cactus products}\label{section_moduli}
In the category of pointed spaces the analogue of the $n$-fold free product of groups is the $n$-fold wedge sum; this takes $n$ pointed spaces and identifies their basepoints.
\begin{equation*}
\begin{array}{l}\includegraphics[height=13mm]{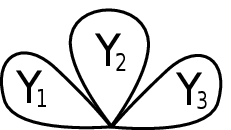}\end{array}
\end{equation*}
The fundamental group of a wedge sum is the free product of the fundamental groups of the summands.
The assumption that the spaces have the homotopy type of CW~complexes means that changing the basepoint does not alter the homotopy type, so forming the wedge sum with a different choice of basepoints gives the same homotopy type.
The moduli space of cactus products defined below describes a space of possible ways to `wedge' $n$ spaces together.

\subsection{Cactus products}
Let $t$ be a tree with vertex set $\sbrac{n}=\cbrac{1,\ldots,n}$.  A tree $t$ with a chosen vertex $i$ is called a \emph{rooted tree}.
The edges of a rooted tree may be oriented by pointing them towards the root, for an edge $e$ the source vertex is denoted $s(e)$ and the target vertex $t(e)$. An example of a rooted tree with root $4$:
\begin{equation}\label{eq_tree}
\xytree{ & 2 \ar[dr] && 1\ar[dl] \\
5\ar[dr]&&3\ar[dl]&\\
&4&.&}
\end{equation}
Now let $(Y_i,\ast_i)$ be pointed spaces for $i=1\clc n$.
A \emph{cactus diagram} $T$ over rooted tree $t$ consists of an edge labelling of $t$: to each edge $e$ there is a label $y_e\in Y_{t(e)}$.
A cactus diagram $T$ gives a \emph{cactus product space} $Y_T$:
\begin{equation}
\frac{Y_1\amalg\ldots\amalg Y_n}{\ast_{s(e)}\sim y_e \mid e\in t}.
\end{equation}
An example corresponding to the rooted tree~\eqref{eq_tree} above:

\begin{equation}
\begin{array}{l}\includegraphics[height=25mm]{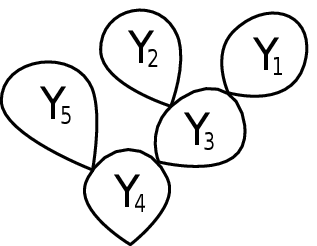}\end{array}
\end{equation}

\noindent If each pointed space is homotopy equivalent to a CW~complex then the cactus product space is homotopic to the wedge product $Y_1\vee\ldots\vee Y_n$.
A cactus product space has a canonical choice of basepoint given by the basepoint of the space corresponding to the root of the tree.
For each $i$ there is a natural inclusion $Y_i\hookrightarrow Y_T$, but note that this inclusion does not respect the basepoint.  The \emph{cactus product} consists of the cactus product space in the context of the diagram
\begin{equation}\xytree{ & Y_T &\\
Y_1\ar[ur] &\cdots\cdots& Y_n.\ar[ul]}\end{equation}

\subsection{The moduli space of cactus products}
A \emph{congruence} of cactus products is a map $Y_T\rightarrow Y_{T'}$ making the diagram
\begin{equation}\xytree{Y_T \ar[rr] & & Y_{T'} \\
 & Y_i \ar[ur]\ar[ul] &} \end{equation}
commute for each $i$.
For example the two decorated trees
\begin{equation}\xytree{ 2\ar[dr]_y & & 3\ar[dl]^y\\ & 1 & } \quad\text{ and }\quad\xytree{3\ar[d]^\ast\\ 2\ar[d]^y\\ 1}\end{equation}
give the same space:

\begin{equation}
\begin{array}{l}\includegraphics[height=22mm]{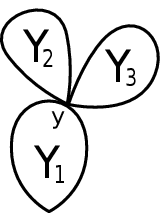}\end{array}
\end{equation}

\begin{lemma}
For a fixed tree $t$, a cactus product $Y_T$ determines the labelling $T$.
And so the set of cactus products over $t$ is naturally identified with 
\begin{equation}\prod_{e} Y_{t(e)}. \label{eq_prodtop}\end{equation}
\end{lemma}
\begin{proof}
Given an edge $e$ the intersection
\begin{equation} Y_{s(e)}\cap Y_{t(e)} \hookrightarrow Y_{t(e)}. \end{equation}
consists of a single point, which is the label $y_e\in Y_{t(e)}$.
\end{proof}
The \emph{moduli space of cactus products} $\MY$ is defined to be the set of cactus products modulo congruence, with the topology inherited from the CW product topology of~\eqref{eq_prodtop}.
It has a canonical basepoint given by the wedge product, which is realised by any tree with every label a basepoint.
\begin{proposition}\label{prop_moduliembedding}
The moduli space of cactus graphs $\MY$ may be embedded into 
\begin{equation} Y_1^{n-1}\tlt Y_n^{n-1}.\end{equation}
\end{proposition}
\begin{proof}
A coordinate in the product space is denoted
\begin{equation} (y_{ij})_{i\neq j}, \end{equation}
where $i,j=1\clc n$ and $y_{ij}$ is an element of $Y_i$. 
Let $T$ be a cactus diagram over a tree $t$.
For each pair $(i,j)$ there is a unique undirected path in $t$ from $j$ to $i$.
If this path is directed, that is if each edge points towards the root, then define $y_{ij}$ to be $y_e$ where $e$ is the last edge in the path (and so has end vertex $i$).
If the path is not directed then let $y_{ij}$ be the basepoint $\ast\in Y_i$.

This map is both well-defined and injective.
\end{proof}

\subsection{Action on $G_1\ala G_n$}
A path in the moduli space of cactus products of pointed spaces $Y_1\clc Y_n$ can be seen to act on the free product of groups
\begin{equation}\label{eq_freefundgps}\pi_1Y_1\ala\pi_1Y_n.\end{equation}
This is illustrated by the following diagram which shows a path representing the automorphism $\alpha_2^{g}$.
\begin{equation}
\begin{array}{l}\includegraphics[height=22mm]{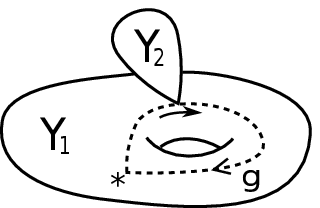}\end{array}
\end{equation}
Any loop $\gamma$ contained in $Y_2$ is taken to a loop which first follows $g$ in $Y_1$ and then follows $\gamma$ before returning around $g$ in the opposite direction.
We will not make this action precise now, instead we compute the fundamental group of $\MY$ and show that it is isomorphic to $\FR(G)$, where $G$ is the free product of the fundamental groups of the pointed spaces~\eqref{eq_freefundgps}.
This computation is a result of Theorems~\ref{thm_YGFFR}~and~\ref{thm_MYYGFiso}.

\subsection{Relationship with Outer space}
Recall that Outer space is the moduli space of marked metric graphs.
It was first defined in~\cite{Culler1986}, see~\cite{Vogtmann2002} for a more recent introduction.
Autre space is similarly defined as the moduli space of marked metric graphs with a chosen basepoint.

Let $Y_i = S^1:=[0,1]/(0\sim 1)$, for $i=1\clc n$.  Then a cactus product is a based graph of rank $n$, which we call a \emph{cactus graph}.  
It is naturally a metric graph with the metric inherited from the copies of $S^1$.
By attaching a marking\footnote{a marking is a homotopy class of based maps from the wedge of $n$ circles, $(R_n,\ast)$} to a cactus graph we obtain a marked metric graph with a basepoint, that is, a point of Autre space.
If we restrict ourselves to a set of markings corresponding to a coset of the pure symmetric automorphisms then this defines an embedding of the universal cover of $\MY$ into Autre space.

A space of marked unpointed cactus graphs was defined in~\cite{Collins1989}, the contractability of the space was shown, along with the result that the outer symmetric automorphism group $O\Sigma(F_n)=\SAut(F_n)/\Inn(F_n)$ acts with finite stabilisers.
However the fact that the outer pure symmetric automorphism group $\FR(F_n)/\Inn(F_n)$ acts properly to give a classifying space was not used.

\subsubsection{Relationship with the Outer space of a free product}
Guirardel and Levitt~\cite{Guirardel2007} defined an `outer space' for any free product of groups $G=G_1\ala G_p\ast F_r$, where each $G_i$ is freely indecomposable and not infinite cyclic.
When $r=0$ their Outer space becomes the space of McCullough and Miller~\cite{McCullough1996}.
In the $r=0$ case the action of $\FR(G)$ on the universal cover of $\MY$ may be extended to an action of the whole automorphism group $\Aut(G)$.  
Since $\MY$ is a classifying space for $\FR(G)$ there is an $\FR(G)$-equivariant map from McCullough-Miller space to the universal cover of $\MY$, it seems likely that this can be extended to an $\Aut(G)$-equivariant map but we do not address that here.
In future work we intend to extend the definition of $\MY$ to allow for any free product of groups possibly including factors isomorphic to $\ZZ$.
Such a space would have small stabilisers in the spirit of the Outer space of a free group.

\section{Diagonal complexes}\label{section_dcs}
Given a graph $\Gamma=(V,E)$ with a group $G_v$ attached to each vertex, the graph product is defined to be the free product of the vertex groups with additional relations for each $vw\in E$ which assert that the elements of $G_v$ commute with those of $G_w$.

In this section we construct a theory which allows for a greater range of commutation relations allowing not only for commutator brackets between the generating elements, but also between certain products of generators.
The graph is replaced by a `diagonal complex' and the graph product by a `diagonal complex product'.

\subsection{Diagonal complexes}
\subsubsection{Partitions}\label{section_partitions}
Let $X$ be a finite non-empty set.  Then a \emph{partition of $X$} is a set of subsets $\cbrac{U_1\clc U_k}$ which are non-empty, pairwise disjoint and whose union is $X$.
We say that a partition is \emph{proper} if $k$ is larger than one.
A \emph{partial partition of $X$} is a set of subsets $\cbrac{U_1\clc U_k}$ which are non-empty and pairwise disjoint, but without the condition that their union is $X$.
We will refer to a subset of a partial partition of $\cUi$ as a \emph{subpartition of $\cUi$}.
Suppose that $\cbrac{U_1\clc U_k}$ and $\cbrac{V_1\clc V_l}$ are two partial partitions of $X$.
Then we say that $\cUi$ is a \emph{coarsening} of $\cVj$ (or that $\cVj$ is a refinement of $\cUi$) if the unions are equal and if for each $V_j$ there is a $U_i$ containing $V_j$.
If $\cUi$ is a coarsening of a subpartition of $\cVj$ then we say that $\cUi$ is a \emph{partial coarsening} of $\cVj$.
The relation of being a partial coarsening is reflexive, antisymmetric and transitive, so defines a poset $(\mathcal{X},\leq_{pc})$ on the set of partial partitions of $X$. 

\subsubsection{Diagonals of a product}\label{section_productdiagonals}
Let $(Y,\ast)$ be a pointed topological space.  
The diagonal map $D_Y:Y\rightarrow Y\times Y$ sends $y$ to $(y,y)$.
If $U\subseteq \sbrac{n}=\cbrac{1,\ldots,n}$ then there is a diagonal map $D^U:Y\rightarrow Y^{\times n}$ given by
\begin{equation}D^U(y) = (x_1,\ldots x_n),\end{equation}
where
\begin{equation}x_i = \begin{cases} y & \text{ if }i\in U, \\
\ast & \text{ if }i\notin U. \end{cases}\end{equation}
Similarly, if $\cbrac{U_1,\ldots,U_k} \subseteq \sbrac{n}$ is a partial partition then we define the map $D^{\cUi}:Y^{\times k}\rightarrow Y^{\times n}$ by
\begin{equation}\label{eq_diagonalmap}
D^{\cUi}(y_1,\ldots,y_k) = (x_1,\ldots x_n),\end{equation}
where
\begin{equation}x_i = \begin{cases} y_j & \text{ if }i\in U_j, \\
\ast & \text{ if }i\notin U_j \text{ for each }j=1,\ldots,k. \end{cases}\end{equation}
This has image
\begin{equation}\bigr\{(y_i)_{i\in \sbrac{n}}\mid y_i=y_j \text{ for $i,j\in U_i$ and } y_l=\ast \text{ for }l\notin \bigcup_i U_i\bigr\}.\end{equation}
From this we may see that the image of $D^{\cUi}$ is contained in the image of $D^{\cVj}$ if and only if $\cUi$ is a partial coarsening of $\cVj$.

We wish to calculate the intersection of the images of $D^\cUi$ and $D^\cVj$ for any two partial partitions.
To do so we define an equivalence relation on $X\cup\cbrac{0}$.  If both $k,l\in U_i$ or $k,l\in V_j$ then we let $k\sim l$.  
If $k\notin \bigcup_i U_i \cap \bigcup_j V_j$ then let $k\sim 0$.  We then extend $\sim$ transitively.
Now let $\cbrac{W_k}$ be the equivalence classes of $\sim$ which do not include $0$.
If this isn't empty then we have 
\begin{equation}\text{Im}D^{\cUi}\cap\text{Im}D^{\cVj} = \text{Im}D^{\cbrac{W_k}}. \end{equation}
If $\cbrac{W_k}$ is empty, that is every $i\in X$ is equivalent to $0$, then the intersection is just $\cbrac{\ast}$.
Recall from Section~\ref{section_partitions} that $\mathcal{X}$ is the poset consisting of partial partitions $\cUi$ and ordered by partial coarsening.
For $\cUi$ and $\cVj$ in $\mathcal{X}$, the partial partition $\cbrac{W_k}$ defined above (if it exists) is the greatest lower bound, or meet.

\subsubsection{Motivation}
Suppose that $(Y,\ast)$ is a pointed space.
In $Y\times Y$ there are three canonical embeddings of $Y$
\begin{equation} y\mapsto (\ast,y),\quad y\mapsto (y,\ast)\quad\text{ and }\quad y\mapsto (y,y),\end{equation}
the final map is referred to as the \emph{diagonal map of $Y$}.
We saw in Section~\ref{section_productdiagonals} that the set of canonical embeddings of $Y^k$ into $Y^n$ are given by the partial partitions $\cbrac{U_1\clc U_k}$ of $[n]=\cbrac{1\clc n}$.

The spaces that we are interested in are unions of these canonical subspaces.
We desire a combinatorial structure that will allow us to describe and manipulate these subspaces in an efficient and concise manner.
The obvious choice might be a subposet with meets of $(\mathcal{X},\leq_{pc})$, this would give us a notion of inclusion and intersection.
However these posets can be unwieldy to work with.
The choice of the structure of a diagonal complex does not allow us to define any subspace of our choosing, but the subspaces it does allow us to define have nice homotopical and group theoretic properties, as will be witnessed by Theorems~\ref{thm_Gdc} and \ref{thm_Hdc}.

\subsubsection{The definition}
Let $X$ be a set, write $X^+=\cbrac{\cbrac{x}\mid x\in X}$ for the set of singleton subsets of $X$.  
The set of finite, non-empty subsets of $X$ is denoted $\Pf X$.
\begin{definition}
A \emph{diagonal complex on $X$} consists of a pair $(\Gamma,\gamma)$, with $\Gamma \subseteq \Pf X$ and $\gamma:\Gamma\rightarrow \Pf \Gamma$ such that
\begin{enumerate}
\item $X^+ \subseteq \Gamma$,
\item for each $U \in \Gamma$ the set $\gamma(U)$ is a partition of $U$ and if $U$ is not a singleton then $\gamma(U)$ is a proper partition.  That is, if $U\in \Gamma - X^+$ then $\mbrac{\gamma(U)}> 1$,
\item (\emph{simplicial condition}) for $U\in \Gamma$ we write $\gamma(U) = \cbrac{U_1,\ldots,U_k}$.  For each $A\subseteq\sbrac{k}$ we require that $U_A\in\Gamma$, where
\begin{equation}U_A:=\bigcup_{i\in A} U_i.\end{equation}
We also require that $\gamma(U_A)$ is either $\cbrac{U_i\mid i\in A}$ or a refinement of $\cbrac{U_i\mid i\in A}$.
We call the $U_A$ the \emph{faces} of $U$.
\end{enumerate}
The \emph{dimension} of $U\in\Gamma$ is defined to be $\mbrac{\gamma(U)}$.
Note that the dimension of a face of $U$ may be greater than the dimension of $U$.
 
Consider the poset on $\Gamma$ transitively generated by $V\leq U$ if $V$ is a face of $U$.
In this poset if $V \leq U$ we say that $V$ is a \emph{descendant of $U$}.
For example a face of a face of $U$ is a descendant, but not necessarily a face of $U$.
If a diagonal complex $(\Gamma,\gamma)$ satisfies
\begin{itemize}
\item[(4)] the descendance order on $\Gamma$ agrees with the ordering by inclusion.
\end{itemize}
then we say that $(\Gamma,\gamma)$ is \emph{proper}.
\end{definition}
\begin{example}\label{example_simplicial}
Let $A_\dt$ be a simplicial complex with vertices $X$.  Let $\Gamma = A_\dt$ and define $\gamma(U) = U^+ = \bigl\{x^+=\cbrac{x}\mid x\in U\bigr\}$.
Then $(\Gamma, \gamma)$ is a proper diagonal complex.
\end{example}
\begin{example}\label{example_T}
Let $X=\sbrac{3}$, we define $\Gamma$ to be the set $\cbbrac{\cbrac{1,2,3}=X, \cbrac{1,2}} \cup X^+$ and $\gamma$ is defined by $\gamma(\cbrac{1,2,3}) = \cbbrac{\cbrac{1,2}, 3^+}$ and $\gamma(\cbrac{1,2}) = \cbrac{1^+,2^+}$.  We may represent $\Gam{}$ by the diagram
\begin{equation}\xy
(0,0)*+{3}="3";
(-5,-8)*+{1}="1";
(5,-8)*+{2.}="2";
(0,-8)*{}="p";
"3";"p"**\dir{-};
"1";"2"**\dir{-};
\endxy\end{equation}
In Section~\ref{section_dcrealisation} this picture will be realised as the geometric realisation of $(\Gamma,\gamma)$.
The descendance order agrees with the inclusion order and so $\Gamma$ is proper.
\end{example}
\begin{example}
An example of a diagonal complex which is not proper is given by adding the element $\cbrac{1,3}$ onto the previous example.
The picture we now get is
\begin{equation}\xy
(0,0)*+{3}="3";
(-5,-8)*+{1}="1";
(5,-8)*+{2.}="2";
(0,-8)*{}="p";
"3";"p"**\dir{-};
"1";"2"**\dir{-};
"3";"1"**\dir{-};
\endxy\end{equation}
Note that $\cbrac{1,3}$ is a subset of $\cbrac{1,2,3}$ but it is not a descendant and so this diagonal complex is not proper.
\end{example}
\noindent A morphism
\begin{equation}f:(\Gamma,\gamma)\rightarrow(\Gamma',\gamma')\end{equation}
from a diagonal complex $(\Gamma,\gamma)$ on $X$ to a diagonal complex $(\Gamma',\gamma')$ on $X'$ is given by an injective function
\begin{equation}f:X\rightarrow X'\end{equation}
satisfying the following conditions,
\begin{itemize}
\item for each $U \in\Gamma$, the image $f(U)$ is in $\Gamma'$.
\item the induced function $f:\Gamma\rightarrow\Gamma'$ makes the following diagram commute:
\begin{equation}\xymatrix{\Gamma\!\phantom{'}\ar[r]^f \ar[d]^{\gamma\phantom{'}} & \Gamma' \ar[d]^{\gamma'}\\
\Pf\Gamma\!\phantom{'}\ar[r]^{\Pf f} & \Pf \Gamma'.}\end{equation}
\end{itemize}
We will sometimes refer to $(\Gamma,\gamma)$ as a diagonal subcomplex of $(\Gamma',\gamma')$.
In particular if $Y\subseteq X$ and $\Gamma_Y := \Gamma\cap\Pf Y$ then $(\Gamma_Y, \gamma|_{\Gamma_Y})$ is referred to as \emph{the full diagonal subcomplex on $Y$}.

When a diagonal complex is proper it satisfies a concise combinatorial property.  
\begin{proposition}\label{prop_proper}
Let $\Gam{}$ be a diagonal complex on a set $X$.  Then $\Gam{}$ is proper if and only if for each $U\in \Gamma$
\begin{equation}\label{eq_maximals}\gamma(U) = \cbrac{U-M_1,\ldots,U-M_k},\end{equation}
where $M_1,\ldots,M_k$ are the maximal subsets of $U$ in $(\Gamma,\subseteq)$.
\end{proposition}
\begin{proof}
Let $U\in \Gamma$, then $\gamma(U)=\cbrac{U_1,\ldots,U_k}$ is a partition of $U$.  
Recall that the unions of the $U_i$ are called the faces and so the maximal faces are the unions
\begin{equation} \bigcup_{i\neq j}U_i \end{equation}
for each $j$.  But this is just $U-U_j$ and by the definition of the descendance order these are maximal under $U$ in the descendance order.

Now suppose that $\Gam{}$ is proper, then the inclusion order agreeings with the descendance order and the sets $U-U_j$ are maximal subsets of $U$.

Conversely, if~\eqref{eq_maximals} holds then $U-U_j$ are both maximal in the descendents order and as subsets of $U$ for every $U\in\Gamma$.
Such a property of finite posets implies that the two orderings are equal.
\end{proof}
In particular this means that a proper diagonal complex $(\Gamma,\gamma)$ is unique with respect to $\Gamma$.
With this in mind we will sometimes just write $\Gamma$ for $(\Gamma,\gamma)$ when it is proper.

\newcommand{\flev}{\text{lev}_\text{f}}
\newcommand{\clev}{\text{lev}_\text{c}}
\subsubsection{Levels of a diagonal complex}
Let $(\Gamma,\gamma)$ be a diagonal complex on a set $X$.
We will define the \emph{level} $l:\Gamma\rightarrow\NN$ inductively as follows:
\begin{itemize}
\item the level is zero if $U\in X^+$, otherwise
\item the level is defined to be the maximum of the levels of the maximal subfaces of $U$, plus one:
\begin{equation}l(U) = \sup\cbrac{l(U-V)\mid V\in \gamma(U)} + 1. \end{equation}
\end{itemize}
The level is well-defined because for $U\in\Gamma-X^+$, the cardinalities of the elements of $\gamma(U)$ are strictly less than the cardinality of $U$.
The level will be used in inductive arguments.  

For $n\in\NN$ we define $(\Gamma_n,\gamma_n)$ to be the diagonal subcomplex given by elements of level $n$ or below.
This defines a filtration of $(\Gamma,\gamma)$:
\begin{equation} (X^+ = \Gamma_0,\gamma_0) \leq (\Gamma_1,\gamma_1) \leq (\Gamma_2,\gamma_2) \leq \ldots. \end{equation}
The union is the whole of $(\Gamma,\gamma)$.
\begin{remark}
There is also a \emph{coarse level} $l_c:\Gamma\rightarrow\NN$ which is similarly defined with $l_c(X^+)=\cbrac{0}$ and
\begin{equation}l_c(U) = \sup\cbrac{l_c(V)\mid V\in \gamma(U)} + 1. \end{equation}
This defines the \emph{coarse filtration} $(\Gamma^c_i,\gamma^c_i)$.
It is worth noting that the zeroth term of both filtrations consists of just the points of $X$, whilst in the coarse case the first term comes from a simplicial complex $\Gamma^c_1$ and in the regular case the first term $\Gamma_1$ is the $1$-skeleton of $\Gamma^c_1$.
The coarse level will not be used in the sequel although it would be sufficient in some arguments.
\end{remark}

\subsection{Diagonal complex products}
\subsubsection{The diagonal complex product}
The very reason for defining diagonal complexes is to study the diagonal complex products they index and which we define in this section.
In fact Theorem~\ref{thm_aspherical} is a result about the homotopy type of particular diagonal complex product of spaces and Theorem~\ref{thm_homology} is a calculation of the homology of a particular diagonal complex product of groups.

Let $(\cat{C},\otimes,k)$ be a strict symmetric monoidal category where the unit $k$ is an initial object.
We define a \emph{coalgebra} $(Y,\Delta)$ in $\cat{C}$ to be an object $Y$ along with a cocommutative comultiplication $\Delta$, that is a morphism $\Delta:Y\rightarrow Y\otimes Y$ satisfying the coassociativity condition which asserts that the following diagram is commutative:
\begin{equation}
\xymatrix{ Y\ar[r]^\Delta\ar[d]_\Delta & Y\otimes Y\ar[d]^{1_Y\otimes\Delta}\\
Y\otimes Y \ar[r]^{\Delta\otimes 1_Y} & Y\otimes Y\otimes Y}
\end{equation}
and also the cocommutative condition which asserts that $\tau_{Y,Y}\circ\Delta=\Delta$.
We will refer to $\Delta$ as the \emph{diagonal map of $Y$}.

The two examples we use in this paper are the category of pointed spaces, see Section~\ref{section_topologicalconsiderations} and the category of groups.
In both cases every object is a coalgebra in a unique way with the diagonal maps defined pointwise.
Another example is given by the cocommutative Hopf algebras in the category of $k$-algebras.

In defining the diagonal complex product indexed by $\Gam{}$, the first thing to consider are the coalgebras which we wish to take the product over.
If we wish to use the diagonal map with target $Y_i\times Y_j$ then the two coalgebras had better be the same.
To ensure this we label each element of $X$ using a function $l:X\rightarrow Z$ and ask for coalgebras $Y_i$ for $i\in Z$.
We extend this to a function $l:\Pf X\rightarrow Z\cup \cbrac{0}$, by letting $l(U)=i$ if $l(j)=i$ for each $j\in U$ and $l(U)=0$ if $l$ is not constant on $U$.
Let $l:X\rightarrow Z$ and $(\Gamma,\gamma)$ be a diagonal complex on $X$.
Then we call $(\Gamma,\gamma,l)$ a \emph{$Z$-labelled diagonal complex} if for each $U\in\Gamma$ and each $U'\in \gamma(U)$, we have $l(U')\neq 0$.
\begin{remark}
If $Z=\cbrac{1}$ then the constant map $l:X\rightarrow Z$ makes any diagonal complex into a $\cbrac{1}$-labelled diagonal complex.
Given a diagonal complex there is a universal labelling $l:X\rightarrow Z$, in the sense that if $l':X\rightarrow Z'$ is any other labelling then there exists a factorisation
\begin{equation}\label{eq_factorisation}\xymatrix{X\ar[r]^l\ar[dr]^{l'} & Z\ar[d] \\ & Z' .}\end{equation}
\end{remark}

\subsubsection{Preliminary construction}
In the category of pointed spaces there is an explicit construction of the diagonal complex product which does not work for general categories, but we shall give it first.
Let $\bdY=(Y_i)_{i\in Z}$ be a $Z$-tuple of pointed spaces and let $\Gam{}$ be a $Z$-labelled diagonal complex, then define $\bdY^X$ to be
\begin{equation}\prod_{i\in X} Y_{l(i)}.\end{equation}
For each $U\in\Gamma$ there is the map
\begin{equation} i_U: \bdY^U:=\prod_{U'\in\gamma(U)} Y_{l(U')}\rightarrow \bdY^X, \end{equation}
defined analagously to $D^{\cUi}$ in~\eqref{eq_diagonalmap}.
The space $\YpGam{}$ is defined as
\begin{equation}\label{eq_YpGam}\bigcup_{U\in\Gamma} \text{Im}\,i_U. \end{equation}
This will turn out to be isomorphic to the diagonal complex product of $(Y_i)_{i\in Z}$ indexed by $\Gaml{}$ as proved in Proposition~\ref{prop_dcproductcolimit}.

\subsubsection{Definition as a colimit}\label{section_defcolimit}
Let $(Y_i)_{i\in Z}$ be a $Z$-tuple of coalgebras in a suitable category $\cat{C}$ and let $\Gaml{}$ be a $Z$-labelled diagonal complex.
We will now formulate the definition of the diagonal complex product $\YGam{}$ as the colimit of a certain diagram.
Define $\mathcal{X}_l$ to be the subposet of $\mathcal{X}$ which contains only those partial partitions $\cUi$ with $l(U_j)\neq 0$ for each $U_j\in\cUi$.
Let $P_\Gamma$ be the subposet of $\mathcal{X}_l$ which contains each partition $\gamma(U)$ for $U\in\Gamma$ and which is closed under taking meets.
The category $\bG$ associated to $(\Gamma,\gamma)$ is the category of the poset $P_\Gamma$.
A functor $F$ from $\mathbf{\Gamma}$ to the category $\cat{C}$ is given as follows:
\begin{itemize}
\item an object $\cUi$ is taken to $\bigotimes_{i}(Y_{l(U_i)})$,
\item the image of a morphism $\cUi\rightarrow\cVj$ is the product of maps 
\begin{equation}
\bigotimes_{U_i=\amalg_{k\in A}V_k}\Delta_{l(U_i)}^{|A|-1}\otimes \bigotimes_{V_k\not\subseteq \cup_i U_i} p_{l(V_k)}
\end{equation}
consisting of a diagonal product for each $U_i$ and an initial morphism for each $V_k$ not contained in a $U_i$.
\end{itemize}
In the category of pointed spaces or groups the maps are those that take $(y_i)$ to $(y'_j)$, where if $V_j\subseteq U_i$ for some $i$ then $y'_j = y_i \in Y_{l(U_i)} = Y_{l(V_j)}$, otherwise $y'_j = \ast \in Y_{l(V_j)}$.
\begin{definition}
The \emph{diagonal complex product} $\YGam{}$ of a $Z$-tuple of coalgebras $\bdY=(Y_i)_{i\in Z}$ indexed by a $Z$-labelled diagonal complex $\Gaml{}$ is the colimit of $F$.
\end{definition}
Recall~\eqref{eq_YpGam} the definition of $\YpGam{}$.
\begin{proposition}\label{prop_dcproductcolimit}
Let $\cat{C}$ be the category of pointed spaces and $\bdY$ and $\Gaml{}$ be as in the definition above.
Then the diagonal complex product $\YGam{}$ is isomorphic to $\YpGam{}$.
\end{proposition}
\begin{proof}
The product $\YpGam{}$ is defined to be the union of the inclusions 
\begin{equation} i_U=D^{\gamma(U)}:\bdY^U\rightarrow \bdY^X.
\end{equation}
So in particular the images of $i_U$ cover the space $\YpGam{}$.
For $U_1\clc U_k\in\Gamma$ the intersection $\im i_{U_1}\cap\ldots\cap\im i_{U_m}$ is either a point or is given by $\im D^{\cbrac{W_k}}$ where $\cbrac{W_k}$ is the meet of the $\gamma(U_j)$ for $j=1\clc m$.

Hence the $\im i_{U_j}$ form a cover of $\YpGam{}$ and the functor $F:\bG\rightarrow\cbrac{\text{Pointed Spaces}}$ is the diagram consisting of the spaces $\im i_{U_j}$ and their intersections in $\YpGam{}$.
In the situation that all the spaces and maps are CW~complexes this implies that $\YpGam{}\cong \colim F=\YGam{}$ by Lemma~\ref{lemma_amalgamations}, see also~\cite{Hatcher2002} Section $4.\text{G}$.
\end{proof}
\begin{example}\label{example_Tcategory}
We defined a diagonal complex $\Gam{}$ in Example~\ref{example_T}.  The category $\bG$ is
\begin{equation} \vcenter{\xymatrix@C=1.4pc@R=0.4pc{&\dt\ar[dd]&\\ &&\\ &\dt&\\ &&\\
&\dt\ar[d]\ar[uu]&\\ \dt\ar[r]&\dt&\dt\ar[l]}}\end{equation}
and the corresponding diagram for coalgebras $Y_1$ and $Y_2$ is given by
\begin{equation} \vcenter{\xymatrix@C=2.4pc@R=0.4pc{&Y_2\ar[dd]^{p\otimes 1}&\\ &&\\ &Y_1\otimes Y_2&\\ &&\\
&Y_1\ar[d]^\Delta\ar[uu]_{1\otimes p}&\\ Y_1\ar[r]^-{1\otimes p}&Y_1\otimes Y_1 &Y_1\ar[l]_-{p\otimes 1},}}\end{equation}
where the morphisms from the initial object are denoted $p$ and the identity morphisms are denoted $1$.
\end{example}

\subsubsection{Geometric Realisation of a diagonal complex}\label{section_dcrealisation}
Let $(\Gamma,\gamma)$ be a diagonal complex.
With the trivial labelling $l:X\rightarrow \cbrac{1}$, it is made into a $\cbrac{1}$-labelled diagonal complex.
We take $(Y_1,\ast)=(\sbrac{0,1},0)$ the unit interval and write $I\Gam{}$ for the diagonal complex product indexed by $(\Gamma,\gamma,l)$.
This is a subspace of $I^X$ and we define $\mbrac{\Gam{}}$ to be the intersection of $I\Gam{}$ and the space
\begin{equation} \cbbrac{(y_i)_{i\in X}\mid \cbrac{i\mid y_i\neq 0} \text{ if finite and }\sum_i y_i = 1}. \end{equation}

\subsubsection{Levelwise construction of products of pointed spaces}\label{section_levelwiseproduct}
We previously defined the level of a $U\in\Gamma$ and saw that there was a filtration
\begin{equation} (X^+ = \Gamma_0,\gamma_0) \leq (\Gamma_1,\gamma_1) \leq (\Gamma_2,\gamma_2) \leq \ldots. \end{equation}
Our aim now is to describe how the diagonal complex product $\YGam{n}$ may be built from that of $\YGam{n-1}$.
Let $U\in\Gam{n}$ be of level $n$ and write $\gamma(U)=\cbrac{U_1\clc U_k}$.  Then the maximal faces $U-U_i$ are all of level $n-1$ or less.
For each proper subset $A\subsetneq\sbrac{k}$ there is an inclusion
\begin{equation}
 \prod_{i\in A}Y_{l(U_i)}\hookrightarrow \prod_{i\in\sbrac{k}}Y_{l(U_i)}=\bdY^U,
\end{equation}
we call this subspace $\bdY^U_A$ and the union of such subspaces we will denote $\delta\bdY^U$.  This is the subspace consisting of elements of $\bdY^U$ where at least one of the coordinates is equal to $\ast$.
From the definition of a diagonal complex the set $U_A=\bigcup_{i\in A}U_i$ is in $\Gam{n-1}$ and $\gamma(U_A)$ is a refinement of $\cUi_{i\in A}$.  
Therefore there is a map
\begin{equation}
 \bdY^U_A\rightarrow \bdY^{U_A}\rightarrow\YGam{n-1},
\end{equation}
which realises $\bdY^U_A$ as the intersection of $\bdY^U$ and $\bdY^{U_A}$ in $\YGam{n}$.
Taking the union of the $\bdY^U_A$ we get that the intersection of $\bdY^U$ and $\YGam{n-1}$ is $\delta\bdY^U$.
So we have a diagram
\begin{equation}\bdY^U \leftarrow \delta\bdY^U\rightarrow \YGam{n-1}\end{equation}
whose colimit attaches $\bdY^U$ onto $\YGam{n-1}$.
In this way $\YGam{n}$ is given by attaching each $U$ of level $n$.

\subsection{Diagonal complex products of groups}
Let $\Gaml{}$ be a $Z$-labelled diagonal complex, where $Z=\cbrac{1\clc m}$.  Let $\bdG=(G_1\clc G_m)$ be an $m$-tuple of groups.
The functor $F$ from $\bG$, which was defined in Section~\ref{section_defcolimit}, to the category of groups is given as follows
\begin{itemize}
\item an object $\cUi$ is taken to $\prod_{i}(G_{l(U_i)})$.
\item the image of a morphism $\cUi\rightarrow\cVj$ is the map that takes $(g_i)$ to $(g'_j)$, where if $V_j\subseteq U_i$ for some $i$ then $g'_j = g_i \in G_{l(U_i)} = G_{l(V_j)}$, otherwise $g'_j = e \in G_{l(V_j)}$.
\end{itemize}
The diagonal complex product indexed by $\Gaml{}$ of these groups was defined to be $\GGam{}:=\colim F$.
The canonical inclusions of $F(\cUi)$ into $\GGam{}$ are denoted~$i_{\cUi}$.

The fundamental group functor 
\begin{equation}\pi_1:\cbrac{\text{Connected Pointed Spaces}}\rightarrow \cbrac{\text{Groups}}\end{equation}
preserves colimits and finite products, so the fundamental group of the diagonal complex product of pointed spaces $Y_i$ is the diagonal complex product of the groups $\pi_1(Y_i)$.

\begin{theorem}\label{thm_Gdc}
Let $\Gaml{}$ be a $Z$-labelled diagonal complex and let $\bdG=(G_i)_{i\in Z}$ be a $Z$-tuple of groups.  Then $\GGam{}$ is generated by elements
\begin{equation}g_U\end{equation}
where $U\in\Gamma$ is such that $l(U)\neq 0$ and $g\in G_{l(U)}$.  These are subject to the relations
\begin{align}
g_U.h_U &= (gh)_U &&\text{for any $g_U, h_U$}\label{eq_rel1},\\
e_U &= e &&\text{for each $U$}\label{eq_rel2},\\
\sbrac{g_U,h_V} &= e &&\text{for $U,V\in \gamma(W)$ and $W\in\Gamma$}\label{eq_rel3},\\
g_U &= g_{U_1}\ldots g_{U_k} &&\text{where $\gamma(U)=\cbrac{U_1\clc U_k}$.}\label{eq_rel4}
\end{align}
These relations suffice to present $\GGam{}$.
The element $g_U$ is given by
\begin{equation}\label{eq_gu}
g_U = i_{\gamma(U)}\circ\Delta_{G_{l(U)}}^{\dim U-1}(g).
\end{equation}
\end{theorem}
\begin{proof}
To give a presentation of $\bdG\Gam{}$ we look to the levelwise construction.  The group $\bdG\Gam{}$ is the colimit of the diagram.
\begin{equation} \bdG\Gam{0} \rightarrow \bdG\Gam{1} \rightarrow \bdG\Gam{2} \rightarrow \ldots. \end{equation}
So it will suffice to prove the theorem for each level, which we will do by induction.
Note that $\bdG\Gam{0}$ is the free product $\ast_{i\in X} G_i$ and that $\bdG\Gam{1}$ is a graph product of groups.
For the case $n=0$ the theorem can easily be seen to be true.
As with the levelwise construction of spaces in Section~\ref{section_levelwiseproduct}, we construct each successive group by amalgamations,
\begin{equation}\bdG^U \leftarrow \delta\bdG^U\rightarrow \GGam{n-1}\end{equation}
where $U$ is of level $n$, $\gamma(U)=\cbrac{U_1\clc U_k}$, the group $\bdG^U$ is $\prod_{i\in\sbrac{k}} G_{l(U_i)}$ and $\delta\bdG^U$ is the colimit of the diagram consisting of all the groups $\prod_{i\in A}G_{l(U_i)}$ for $A\in\delta\Delta_k$.
However in the category of groups, if the dimension of $U$ is greater than two, then $\delta\bdG^U \cong \bdG^U$.
Only in the case when the dimension of $U$ is two, when 
\begin{equation}\delta\bdG^U = G_{l(U_1)}\ast G_{l(U_2)}\end{equation}
and so the diagram looks like
\begin{equation}\label{eq_levelcommutator}
G_{l(U_1)}\times G_{l(U_2)} \leftarrow G_{l(U_1)}\ast G_{l(U_2)}\rightarrow \GGam{n-1}.
\end{equation}
does taking the colimit of the diagram have any effect.  The effect in question is that of adding commutation relations.
Even though the amalgamation may not change the group it is still a good time to prove that the relations~\eqref{eq_rel1}-\eqref{eq_rel4} hold for $U$.
The group $\bdG^U$ is a copy of $\prod_{U'\in\gamma(U)}G_{l(U')}$ and each of the factors are seen to be one of the $F(\cbrac{U'})$, which map diagonally into $F(\gamma(U'))$.
Hence by~\eqref{eq_gu} each of the factors consists of the elements $g_{U'}$ for $g\in G_{l(U')}$, and so relation~\eqref{eq_rel3} is seen to hold.
Now assume that $l(U)\neq 0$.  Then $g_U$ for each $g$ is given by the diagonal $G_{l(U)} \rightarrow \bdG^U$, hence $g_U = \prod_{U'\in\gamma(U)}g_{U'}$ giving relation~\eqref{eq_rel4}.
Finally relations~\eqref{eq_rel1} and~\eqref{eq_rel2} are also given by the inclusion $G_{l(U)}\rightarrow \bdG^U$.
\end{proof}
Note that the set $\cbrac{g_U\mid U\in X^+}$ generates the group and that the remaining $g_U$ are defined for notational convenience.

\subsubsection{Diagonal right-angled Artin groups}
Let $\Gam{}$ be a diagonal complex and consider it to be $\cbrac{1}$-labelled.
Let $G_1 = \ZZ$, then we call $\GGam{}$ a \emph{diagonal right-angled Artin group}, or DRAAG.
The DRAAGs represent a generalisation of right-angled Artin groups, or RAAGs, see~\cite{Charney2007b} for an introduction to their theory.
One may study the space $\YGam{}$ with $Y_1=S^1$, however it is not clear whether there's a concise combinatorial property which allows us to determine whether $\YGam{}$ is aspherical or not.
Using the approach of CAT(0) geometry is not sufficient; there are spaces $\YGam{}$ which are aspherical but not CAT(0), an example is given by $\bdY\Gf$ defined in Section~\ref{section_forests}. 

\subsection{Homology of diagonal complex products}
Let $Y$ be a pointed space and let $C_\ast(Y)$ be the chain complex of the space $Y$ over a ring $R$.
Then $C_\ast(Y)$ splits as $R\oplus\rC(Y)$, where $\rC(Y)$ is the reduced chain complex.
Recall that there is a decomposition of the homology of a product of pointed spaces:
\begin{proposition}\label{prop_Hdecomp}
Let $Y_i$ be pointed spaces for $i=1\clc n$.  
Then the homology of the direct product of the pointed spaces is given as
\begin{equation}H_\ast\brac{\prod_{i=1}^n Y_i} = R\oplus\bigoplus_{A\in\Delta_n} H_\ast\brac{\bigotimes_{i\in A} \rC(Y_i)},\end{equation}
where $\Delta_n = \Pf[n]$ is the $n$-simplex.
\end{proposition}
\begin{proof}
Using the definition of the homology and a K\"unneth formula we have
\begin{align*}
H_\ast\brac{\prod_{i=1}^n Y_i} &\cong H_\ast\brac{C_\ast\brac{\prod_{i=1}^n Y_i}} \\
&\cong H_\ast\brac{\bigotimes_{i=1}^n C_\ast(Y_i)}.
\end{align*}
By using the decomposition $C_\ast(Y_i)\cong R\oplus \rC(Y_i)$ we have
\begin{align*}
\bigotimes_{i=1}^n C_\ast(Y_i) &\cong \bigotimes_{i=1}^n (R\oplus \rC(Y_i))\\
 &\cong \bigoplus_{A\subseteq\sbrac{n}} \bigotimes_{i\in A} \rC(Y_i).
\end{align*}
We are done because $H_\ast$ commutes with $\oplus$ and the $R$ term comes from $A=\emptyset\subset\sbrac{n}$.
\end{proof}
Using this we get a similar decomposition of the homology of a diagonal complex product.
\begin{theorem}\label{thm_Hdc}
Let $\Gam{}$ be a $Z$-labelled diagonal complex on a set $X$ and $\bdY=(Y_i)_{i\in Z}$ be pointed spaces.
The homology of the diagonal complex product $\YGam{}$ splits as
\begin{equation}\label{eq_Hdecomp}
R\oplus\bigoplus_{U\in\Gamma} H_\ast(U),
\end{equation}
where $H_\ast(U)$ is given by
\begin{equation}H_\ast\Brac{\bigotimes_{U'\in\gamma(U)} \rC(Y_{l(U')})}. \end{equation}
\end{theorem}
\begin{proof}
We proceed by induction on the level.  For the case of $n=0$, the product is a wedge product of the spaces $Y_i$ and so the theorem holds.
Now suppose that~\eqref{eq_Hdecomp} holds for level $n-1$, that is for the diagonal complex $\Gam{n-1}$.
In Section~\ref{section_levelwiseproduct} the space $\YGam{n}$ was formed by gluing copies of $\bdY^U$ onto $\YGam{n-1}$ using diagrams of the form
\begin{equation}\label{eq_Hgluing}
\YGam{n-1} \leftarrow \delta\bdY^U \rightarrow \bdY^U.
\end{equation}
Write $\gamma(U)=\cbrac{U_1\clc U_k}$.
Using Proposition~\ref{prop_Hdecomp} we are able to write 
\begin{equation}C_\ast(\delta\bdY^U)\rightarrow C_\ast(\bdY^U)\end{equation}
as a split monomorphism
\begin{equation*}
R\oplus \bigoplus_{A\in\delta\Delta_k} \bigotimes_{i\in A} \rC(Y_{l(U_i)}) \hookrightarrow
R\oplus \bigoplus_{A\in\Delta_k} \bigotimes_{i\in A} \rC(Y_{l(U_i)}).
\end{equation*}
For each $U_i$ there is a map $Y_{l(U_i)} \rightarrow \bdY^{U_i}\cong Y_{l(U_i)}^{\dim U_i}$ given by the diagonal map, and this induces the inclusion $\delta\bdY^U \hookrightarrow \YGam{n-1}$. 
On chains the map $Y_{l(U_i)} \rightarrow \bdY^{U_i}$ induces an injection 
\begin{equation} C_\ast(Y_{l(U_i)}) \rightarrow C_\ast(\bdY^{U_i}) \end{equation}
and so the map $C_\ast(\delta\bdY^U)\rightarrow C_\ast(\YGam{n-1})$ is injective.
Taking the colimit of~\eqref{eq_Hgluing} has the effect of adding on a term 
\begin{equation}\bigotimes_{i=1}^k \rC(Y_{l(U_i)}),\end{equation}
which gives the term $H_\ast(U)$ in~\eqref{eq_Hdecomp}.  Thus we have proved the theorem.
\end{proof}
\begin{remark}[cohomological version of Theorem~\ref{thm_Hdc}]\label{rem_coHdc}
The cohomological version is similar, although because taking cochains is contravariant, the colimit is replaced by a limit and so the direct sum should be replaced by the direct product.
In the case that $\Gamma$ and $X$ are finite the direct product is isomorphic to the direct sum and so we have that the cohomology of the diagonal complex product $\YGam{}$ splits as
\begin{equation}\label{eq_coHdecomp}
R\oplus\bigoplus_{U\in\Gamma} H^\ast(U),
\end{equation}
where $H^\ast(U)$ is given by
\begin{equation}H^\ast\Brac{\bigotimes_{U'\in\gamma(U)} \rcC(Y_{l(U')})} \end{equation}
and
\begin{equation}C^\ast(Y_{l(U')}, R) \cong R\oplus\rcC(Y_{l(U')}). \end{equation}
\end{remark}
\begin{remark}
It is worth noting that the proof of the theorem above offers more than a calculation of the homology of diagonal complex products, there is also a natural quasi-isomorphism realising this equivalence
\begin{equation}\label{eq_quasiiso}
R\oplus \bigoplus_{U\in\Gamma} \biggl(\bigotimes_{U'\in\gamma(U)} \rC(Y_{l(U')})\biggr)\rightarrow C_\ast\bigl(\bdY\Gam{}\bigr).
\end{equation}
The importance of this will become clear in Theorem~\ref{thm_homotopyquotient}.  There is also a cohomological version, where of course the arrow is reversed
\begin{equation}
R\oplus \bigoplus_{U\in\Gamma} \biggl(\bigotimes_{U'\in\gamma(U)} \rcC(Y_{l(U')})\biggr)\leftarrow C^\ast\bigl(\bdY\Gam{}\bigr).
\end{equation}
\end{remark}

\subsubsection{Hilbert-Poincar\'e series}
The theorem motives looking at the Hilbert-Poincar\'e series of a $Z$-labelled diagonal complex $\Gaml{}$.
To each simplex $U\in\Gamma$ we assign a monomial in the elements of $Z=\cbrac{z_1\clc z_k}$,
\begin{equation}\label{eq_monomial} m(U) = \prod_{U'\in\gamma(U)} l(U').\end{equation}
The Hilbert-Poincar\'e series of $\Gam{}$ in the polynomial ring $\ZZ[Z]$ is
\begin{equation}h_{\Gam{}}(z_1\clc z_k) = \sum_{U\in\Gamma} m(U). \end{equation}
Let $\mathcal{R}\ttt$ be the complete $\NN$-graded ring generated by the indecomposable modules of $R$ with the product $\Tor^R$.
We use the variable $t$ to keep track of the grading.  For example if $R=\ZZ$ then the indecomposable modules are of the form $\ZZ$ and $\ZZ/(p^i)$ for a prime $p$.
The module $\ZZ$ is the unit of $\Tor^\ZZ$ and we will denote $\ZZ/(p^i)$ by $x_{p^i}$.
Then
\begin{equation}\label{eq_tormult} x_{p^i}.x_{q^j} = \begin{cases}
(1+t)x_{p^i} & \text{ if $p=q$ and $i\leq j$,}\\
0 & \text{ if $p\neq q$}
\end{cases}\end{equation}
is used to denote the fact that
\begin{equation}
\Tor^\ZZ_k \brac{\ZZ/(p^i),\ZZ/(q^j)} \cong\begin{cases}
\ZZ/(p^i) & \text{ if $p=q$, $i\leq j$ and $k=0,1$},\\
0 & \text{ if $p\neq q$ and $k=0,1$,}\\
0 & \text{ if $k\geq 2$}.
\end{cases}
\end{equation}
Now let $y_i(t)$ and $y'_i(t)$ in $\mathcal{R}\ttt$ be the Hilbert-Poincar\'e series of $H_\ast(Y_i)$ and $H^\ast(Y_i)$ respectively.
\begin{corollary}\label{cor_hilb}
The Hilbert-Poincar\'e series of the homology and the cohomology of the diagonal complex product $\bdY\Gam{}$ are respectively
\begin{equation*}1+h_{\Gam{}}(y_1-1\clc y_k-1)\quad\text{ and }\quad 1+h_{\Gam{}}(y'_1-1\clc y'_k-1). \end{equation*}
\end{corollary}

\subsubsection{The cohomology ring}
We may use the inclusion $\YGam{}\hookrightarrow \bdY^X$ to calculate the cup product on cohomology.
\begin{lemma}\label{lem_inj}
The map $H^\ast\bbrac{\bdY^X,R}\rightarrow H^\ast\bbrac{\YGam{},R}$ is surjective.
\end{lemma}
\begin{proof}
By the remarks following Theorem~\ref{thm_Hdc} the cohomology $H^\ast\bbrac{\YGam{},R}$ decomposes as
\begin{equation}\label{eq_leminj1} R\oplus \bigoplus_{U\in\Gamma} H^\ast\brac{\rcC(U)}. \end{equation}
And by a cohomological version of Proposition~\ref{prop_Hdecomp} the complex $H^\ast(\bdY^X,R)$ decomposes as
\begin{equation}\label{eq_leminj2} R\oplus \bigoplus_{A\in\Delta_X} H^\ast\Brac{\bigotimes_{i\in A} \rcC(Y_{l(i)})}. \end{equation}
The inclusion of $\bdY^U=\prod_{U'\in\gamma(U)}Y_{l(U')}$ into $\prod_{x\in U}Y_{l(x)}$ by diagonal maps induces a map
\begin{equation}\label{eq_leminj3} H^\ast\Brac{\bigotimes_{x\in U}C^\ast(Y_{l(x)})}\rightarrow H^\ast\Brac{\bigotimes_{U'\in \gamma(U)}C^\ast(Y_{l(U')})}\rightarrow H^\ast\brac{\rcC(U)}.\end{equation}
So each summand in~\eqref{eq_leminj1} is mapped to from a summand in~\eqref{eq_leminj2}.
It now only remains to note that since there is a projection from $\bdY^X$ to $\bdY^U$ providing a one-sided inverse to the inclusion, the maps~\eqref{eq_leminj3} on the individual summands are surjective.
\end{proof}
The product can now be calculated by considering the diagram
\begin{equation} \xymatrix{
 H^\ast\bbrac{\YGam{},R}\otimes H^\ast\bbrac{\YGam{},R} \ar[d] && H^\ast\bbrac{\bdY^X,R}\otimes H^\ast\bbrac{\bdY^X,R} \ar@{->>}[ll]\ar[d]\\
 H^\ast\bbrac{\YGam{},R} && H^\ast\bbrac{\bdY^X,R}.\ar@{->>}[ll]
}\end{equation} 
Given cocycles in $H^\ast\bbrac{\YGam{},R}$ to multiply, first lift along the surjective map to cocycles in $H^\ast\bbrac{\bdY^X,R}$, multiply them there, then map back to $H^\ast\bbrac{\YGam{},R}$.

\subsubsection{Homology of homotopy quotients}
Let $\Gaml{}$ be a $Z$-labelled diagonal complex and let $\bdY$ be a $Z$-tuple of pointed spaces.
Furthermore let $l':X\rightarrow Z'$ be a finer labelling, in that there exists a factorisation
\begin{equation}\xymatrix{X\ar[r]^{l'}\ar[dr]^{l} & Z'\ar[d]_{\pi} \\ & Z .}\end{equation}
The finer labelling could for example be the universal labelling.
Let $\bdH$ be a $Z$-tuple of groups such that $H_i$ acts on $Y_i$ by basepoint fixing homeomorphisms and let $\sym$ be a group of label fixing automorphisms of $\Gaml{}$.
Both the direct product 
\begin{equation} H:= \prod_{i\in Z'} H_{\pi(i)} \end{equation}
and $\sym$ then act on the diagonal complex product $\bdY\Gam{}$ together giving an action of the semidirect product $H\rtimes \sym$.
The action by $H_i$ on $\bdY\Gam{}$ is via a diagonal action
\begin{equation}\label{eq_Haction}
h.y_x = \begin{cases} (h.y)_x &\text{ if $l'(x)=i$ and} \\ y_x&\text{ otherwise.}
\end{cases}\end{equation}

Let $E_G$ be a contractible space on which $G$ acts properly and freely.
The homotopy quotient of an action of a group $G$ on a space $X$ may be given by
\begin{equation}
X\times_G E_G = (X\times E_G)/G,
\end{equation}
the quotient of the direct product of $X$ and $E_G$ where $G$ acts diagonally.
We are interested in the homotopy quotient of $\bdY\Gam{}$ by the action of $H\rtimes \sym$.
\begin{theorem}\label{thm_homotopyquotient}
Let $\Gaml{}$, $\bdY$, $\bdH$ and $\sym$ be as above.  Then the homology of the homotopy quotient of $\bdY\Gam{}$ by $H\rtimes\sym$ with coefficients in $R$ splits as
\begin{equation}
H_\ast(H\rtimes\sym, R) \oplus \bigoplus_{\sbrac{U}\in\Gamma/\sym} H_\ast\bigl(H\rtimes \text{Stab}_\sym(U),\rC(U)\bigr),
\end{equation}
where 
\begin{equation}
\rC(U):=\bigotimes_{U'\in\gamma(U)} \rC(Y_{l(U')})
\end{equation}
has the permutation action by $\text{Stab}_\sym(U)$.  The action of $H$ on $C_\ast\bigl(\bdY\Gam{}\bigr)$ restricts to $\rC(U)$.
\end{theorem}
\begin{proof}
The action of $\sym$ on $\Gamma$ permutes the $U\in\Gamma$ and so permutes the terms in the direct sum~\eqref{eq_quasiiso}, thus making the quasi-isomorphism~\eqref{eq_quasiiso} into a morphism of $\sym$-modules.
Since the action of $H$ on $C_\ast\bigl(\bdY\Gam{}\bigr)$ restricts to each term
\begin{equation}
\bigotimes_{U'\in\gamma(U)}\rC(Y_{l(U')}),
\end{equation}
the morphism from~\eqref{eq_quasiiso} is in fact a quasi-isomorphism of $H\rtimes\sym$-modules.
The homology of the homotopy quotient of $\bdY\Gam{}$ may be calculated by the homology of $H\rtimes \sym$ with coefficients in $C_\ast\bigr(\bdY\Gam{}\bigr)$ so we can compute it by calculating the homology of $H\rtimes\sym$ with coefficients in the decomposition
\begin{equation}\label{eq_decompHsym}
R\oplus\bigoplus_{U\in\Gamma}\biggl(\bigotimes_{U'\in\gamma(U)}\rC(Y_{l(U')})\biggr).
\end{equation}
The direct sum
\begin{equation}
\bigoplus_{U\in\Gamma}
\end{equation}
factors into
\begin{equation}
\bigoplus_{\sbrac{U}\in\Gamma/\sym} \bigoplus_{V\in \sbrac{U}}
\end{equation}
and so the $H\rtimes\sym$-module~\eqref{eq_decompHsym} decomposes over the sum
\begin{equation}
\bigoplus_{\sbrac{U}\in\Gamma/\sym}{}.
\end{equation}
Now we may concentrate on individual terms in this decomposition.  A term
\begin{equation}
\bigoplus_{V\in \sbrac{U}}\bigotimes_{U'\in\gamma(U)}\rC(Y_{l(U')})
\end{equation}
may be given by inducing the $H\rtimes\text{Stab}_\sym(U)$-module $\rC(U)$ up to $H\rtimes\sym$.
So by the Shapiro lemma the homology of the corresponding term is given by the homology of $H\rtimes\text{Stab}_\sym(U)$ with coefficients in $\rC(U)$.
This completes the proof.
\end{proof}
\begin{example}
A special case of this theorem is when each the diagonal complex comes from a full simplicial complex $\Delta_n=\Pf\sbrac{n}$ and so the diagonal complex product is the direct product.
With the trivial labelling this has automorphism group $\symn$.
So for a space $Y$ with $H$ the trivial group, the homotopy quotient is $Y^n\times_{\symn} E_{\symn}$.
When $Y$ is a classifying space for a group $G$, this homotopy quotient is a classifying space for the wreath product of $G$ with $\symn$.
The homology of these spaces was studied in~\cite{Leary1997}.
\end{example}

\section{Asphericity and coset complexes}\label{section_cosetcomplexes}
A pointed space $Y$ is called \emph{aspherical} if $\pi_i(Y)=0$ for all $i\geq 2$.
Let $\Gam{}$ be a diagonal complex.
Then we call $\Gam{}$ \emph{aspherical} if for every labelling $l:X\rightarrow Z$ and every $Z$-tuple of aspherical pointed spaces $\bdY=(Y_i)_{i\in Z}$, the diagonal complex product $\YGam{}$ is aspherical.

When $\YGam{}$ is defined by a simplicial complex $A_\dt$ as in Example~\ref{example_simplicial} then $\YGam{}$ is aspherical if and only if $A_\dt$ is a flag complex.
For the definition of a flag complex and a discussion of the condition for RAAGs see~\cite{Charney2007b}.

It would be helpful to have a general combinatorial condition that would tell us when a diagonal complex was aspherical.
A major step in this direction would be an answer to the following.
\begin{question}
Let $\Gamma$ be a proper diagonal complex on a set $X$.  Suppose that $X\in\Gamma$, then is $\Gamma$ aspherical?
\end{question}
A positive answer to the next question would allow geometric techniques to be applied to the diagonal complex product of the circle. 
\begin{question}
Let $\Gamma$ be a proper diagonal complex on a set $X$ and let $Y_1=S^1$.  If $\YGam{}$ is aspherical, does this imply that $\Gamma$ is aspherical?
\end{question}

\subsection{Coset complexes}
Our approach to the problem of showing that examples of diagonal complexes are aspherical is to study certain coset complexes associated to natural subgroups.
Recall that the diagonal complex product of groups is defined as the colimit from a certain category $\bG$.
Each object $\cUi$ of $\bG$ determines a subgroup
\begin{equation}
\prod_{U_j\in\cUi} G_{l(U_j)} \hookrightarrow \GGam{},
\end{equation}
the fact that this morphism is injective can be seen by taking the composition
\begin{equation}
\prod_{U_j\in\cUi} G_{l(U_j)} \rightarrow \GGam{} \rightarrow \prod_{x\in X}G_{l(x)},
\end{equation}
which is a diagonal map and hence injective.
The category $\bG$ was defined as a subposet closed under meets of the partial partitions of $X$ under partial coarsening.
The meet in this poset corresponds to the intersection of subgroups in $\GGam{}$, so the collection of subgroups parametrised by $\bG$ is closed under intersections.

We now recall some material from the paper~\cite{Abels1993} on which our treatment of the coset complex is based.
Let $G$ be a group with a finite family of subgroups $\HH=\cbrac{H_j\mid j\in J}$ closed under intersection.
For such a group let $\hh$ be the set of cosets 
\begin{equation}
\coprod_{j\in J} G/H_j.
\end{equation}
Since $\HH$ is closed under taking intersections and since a non-empty intersection of cosets $g_1H_1 \cap g_2H_2$ is a coset of the intersection of the respective subgroups $u(H_1\!\cap\! H_2)$, the cosets $\hh$ are closed under taking non-empty intersections.
The set $\hh$ may be viewed as a cover of $G$.

Let $X$ be a set and $\UU$ be a covering of that set, we will assume that $\UU$ is closed under taking non-empty intersections.
Under inclusion $\UU$ forms a poset $(\UU,\subseteq)$, which has a nerve $N(\UU,\subseteq)$ which is a simplicial set where the $k$-simplices are chains
\begin{equation}
U_0\subseteq U_1\subseteq\ldots\subseteq U_{k}.
\end{equation}
We also define another simplicial set $\Xs(\UU)$ where the $k$-simplices are $(k+1)$-tuples $(x_0\clc x_k)$ of $X$ such that there is a $U$ in $\UU$ containing each $x_i$ for $i=1\clc k$.
Results from Section~1.6 and Theorem~1.4 from~\cite{Abels1993} imply that $\Xs(\UU)$ and $N(\UU,\subseteq)$ are homotopic.

\begin{definition}
Let $G$ be a group with a finite family of subgroups $\HH=\cbrac{H_j\mid j\in J}$ closed under intersection and let $\hh$ be the associated set of cosets which may be viewed as a cover of $G$.  Then the nerve $N(\hh,\subseteq)$ is called the \emph{coset complex}.
\end{definition}
The functor $B$ assigns to a group $G$ the simplicial set $B(G)$ which is a classifying space for $G$.
Recall that the set of $k$-simplices of $B(G)$ is the set $G^k$.
\begin{theorem}\label{thm_cosetcomplex}
Let $G$ be a group with a finite family of subgroups $\HH=\cbrac{H_j\mid j\in J}$ closed under intersection.
Let $(\hh,\subseteq)$ be the poset of cosets as above.
Then the colimit $\colim_{H\in\HH} B(H)$ is a classifying space for $G$ if and only if the coset complex $N(\hh,\subseteq)$ is contractible.
\end{theorem}
\begin{proof}
For a group $H$, define $\Xs(H)$ to be the simplicial set with $k$-simplices the $(k+1)$-tuples of $H$.  The face maps are given by forgetting a coordinate and the degeneracy maps given by duplicating coordinates.
This has a free $H$-action given by 
\begin{equation}
h.(h_0\clc h_k) = ( hh_0\clc hh_k).
\end{equation}
This construction is functorial so if $H$ is a subgroup of $G$ then there is an inclusion of simplicial sets $\Xs(H)\hookrightarrow \Xs(G)$.
Define $\XsG{H}$ to be the $G$-orbit of the image of $\Xs(H)$ in $\Xs(G)$.  The induction notation is appropriate because both $H$ and $G$ actions are free.
Explicitly the $k$-simplices of $\XsG{H}$ consist of those $(k+1)$-tuples which lie in a single coset $gH$.
Note that if $H \leq H'$ then $\XsG{H} \subseteq \XsG{H'}$ and furthermore
\begin{equation}
\XsG{H_1} \cap \XsG{H_2} = \XsG{H_1\cap H_2}.
\end{equation}
So the colimit $\colim_{H\in\HH}\XsG{H}$ is a subspace of $\Xs(G)$.
This subspace is the span of all of the inclusions of $\XsG{H}$ for $H\in \HH$, so explicitly it consists of $(k+1)$-tuples $(g_0\clc g_k)$ for which there exists a coset $gH$ containing each $g_i$ for $i=0\clc k$.
But this is precisely the space $\Xs(\hh)$.  The space $\Xs(\hh)$ has a free $G$-action and we now take the quotient:
\begin{equation}
\Xs(\hh) / G = \colim_{H\in\HH} \XsG{H}\! / G \cong \colim_{H\in\HH} \Xs(H) / H.
\end{equation}
Since each $\Xs(H)$ is a contractible space with a free $H$-action this means that
\begin{equation}
\Xs(\hh) / G \cong \colim_{H\in\HH} B(H).
\end{equation}
So $\colim_\HH B(H)$ is a classifying space for $G$ if and only if $\Xs(\hh)$ is contractible.  However $\Xs(\hh)$ is homotopic to $N(\hh,\subseteq)$.  
We are done.
\end{proof}

\newcommand{\CCGG}{\text{CC}_\Gamma(\bdG)}
\newcommand{\CCGGf}{\text{CC}_\Gf(\bdG)}
This may be applied to a diagonal complex product as follows.
Let $\Gam{}$ be a $Z$-labelled diagonal complex and let $\bdG$ be a $Z$-tuple of groups.  As discussed above the category $\bG$ parametrises a family of subgroups of $\GGam{}$ which are closed under intersections.
This satisfies the conditions of Theorem~\ref{thm_cosetcomplex} and we write $\CCGG$ for the coset complex of this family.


\section{The diagonal complex of forest posets}\label{section_forests}
We now apply the results of Sections~\ref{section_dcs} and~\ref{section_cosetcomplexes} to study the moduli space of cactus products defined in Section~\ref{section_moduli}.
To this end we describe a diagonal complex of forest posets, $\Gf$ which indexes the moduli space: $\MY\cong\bdY\Gf$.  
Then we prove Theorem~\ref{thm_aspherical} using the coset complex and the contractability of a variant of McCullough-Miller space from~\cite{Chen2005}.

\subsection{The diagonal complex $\Gf$}
Let $(P,\leq)$ be a finite poset.
The \emph{Hasse diagram} of $(P,\leq)$ is the directed graph with vertex set $P$ and an edge $\edij$ when $i<j$ are adjacent, that is, when $i\leq k\leq j$ implies that $k=i$ or $j$.
Conversely a directed graph without directed loops defines a poset given by setting $i<j$ when $\edij$ is an edge and then taking the transitive closure.
A directed graph is called a \emph{planted forest} if for each vertex the number of incoming edges is not greater than one, if there are no cycles and if the edge set is non-empty.
Under the transitive closure the planted forests are taken to posets $(P,\leq)$ with the \emph{underset condition}:
\begin{equation}\label{eq_underset}
 \text{for all }x\in P,\, \cbrac{y\mid y\leq x}\text{ is a total order.}
\end{equation}
In fact the correspondence is a bijection between the set of planted forests and the set of finite non-trivial posets satisfying~\eqref{eq_underset}.
Let $P_n$ be the set $\sbrac{n}$.  A poset $(P_n,\leq)$ defines a set
\begin{equation}\cbrac{(i,j)\mid i<j},\end{equation}
which is a subset of $X_n = P_n\times P_n-\Delta P_n$. 
We write $\Gf\subseteq \Pf X_n$ for the set of subsets of $X_n$ given by forest posets $(P_n,\leq)$.
An uppercase letter such as $U$ will be used to denote the poset $(P_n,\leq)$, the subset of $X_n$ and the corresponding planted forest.
This should not cause confusion.

We will next define a map $\gf:\Gf\rightarrow\Pf\Gf$.
For a forest poset $U$ take $i<j$, then choose a maximal path from $i$ to $j$
\begin{equation}i=i_0< i_1<\ldots < i_m = j.\end{equation}
The pair $(i,i_1)$ necessarily gives an edge $\ed{ii_1}\in U$, write $\mu_U(i,j) = \ed{ii_1}$.
Then $\mu_U$ is a map from the set $U\subseteq X_n$ to $E(U)$, the edge set of the planted forest.
Since the path is within the set $\cbrac{k\mid k\leq j}$ which is a total order, the path is unique and so $\mu_U$ is well-defined.
For an edge $\edij$ we have $\mu_U(i,j)=\edij$, so $\mu_U$ is surjective on the edge set $E(U)$.
The map $\mu_U$ defines a partition of $U$, each subset is given in the form $\mu_U^{-1}(\edij)$ for some edge of $U$, and this defines $\gamma(U)$.
For example
\begin{equation}\gf\brac{\xytree{1&\\2\ar[u]&4\\3\ar[u]\ar[ur]&}} = \cbrac{\xytree{1\\2\ar[u]},\xytree{1&2\\3\ar[u]\ar[ur]&},\xytree{4\\3\ar[u]}}. \end{equation}
\begin{proposition}
The pair $(\Gf,\gf)$ is a diagonal complex on $X_n$.  Furthermore this diagonal complex is proper.
\end{proposition}
\begin{proof}
For each element $(i,j)$ of $X_n=P_n\times P_n-\Delta P_n$, the tree $\xymatrix@1@=1pc{i\ar[r]&j}$ gives the poset given by $\cbrac{(i,j)}$ and so condition (1) in the definition of a diagonal complex is satisfied.
That $\gamma(U)$ is a partition of $U$ is immediate from the definition.
If $\gamma(U)=\cbrac{U}$ then $U$ is a forest with only one edge and so is contained in $X_n^+$, hence condition (2) is satisfied.

Next we must show that $\Gf$ contains the faces of each $U\in\Gf$.  Let $\cUi_{i\in E(U)} = \gamma(U)$ and suppose $Z\subseteq E(U)$.
We write $U_Z$ for $\bigcup_{i\in Z} U_i$.  Then $U_Z$ is the subset of $U$ consisting of $i<_Z j$ when the maximal path
\begin{equation}i=i_0< i_1<\ldots < i_m = j\end{equation}
has $\ed{ii_0}\in Z$.  To check that this is a poset we need only check that it is transitively closed.
This is immeditate because for $i<j<k$ the maximal path from $i$ to $j$ is a subpath of the maximal path from $i$ to $k$, so $i<_Z j$ implies that $i<_Z k$.
We now need to check that each underset in $U_Z$ is a total order.
So suppose that $i\leq_Z k$ and $j\leq_Z k$.
Since both $i<k$ and $j<k$ in $U$ and $U$ satisfies the underset condition, we may assume that $i<j$.
But since the maximal path joining $i$ and $j$ is a subpath of the path joining $i$ and $k$, then $i<_Z j$ and so $U_Z$ satisfies the underset condition and so is in $\Gf$.

For the remainder of condition (3) it is enough to show that each $V\in\gf(U_Z)$ is contained in some $U'\in\gf(U)$.
The set $V$ corresponds to some edge $\edij\in E(U_Z)$, let $\edik$ be the edge $\mu_U(i<j)$ and let $U'=\mu_U^{-1}(\edik)$.
For $i<l$ in $V$, we have $j\leq l$ and so the maximal path joining $i$ and $j$ in $U$ travels through $j$ and hence the maximal path joining $i$ and $j$ is a subpath.
So the maximal path joining $i$ and $l$ in $U$ starts with $\edik$ and hence $\mu_U(i,l) = \edik$ and so $(i,l)\in U'$.
We have shown that $V\subseteq U'$ and have completed the proof that $\Gam{\Forn}$ is a diagonal complex.

Finally to see that $(\Gf,\gf)$ is proper using Proposition~\ref{prop_proper} we must show that for each $U\in\Gf$ the maximal faces $U-\mu_U^{-1}(\ed{ij})$ are maximal subsets of $U$ in $\Gf$.
So suppose that 
\begin{equation}U-\mu_U^{-1}(\ed{ij}) < V \leq U \end{equation}
in $\Gf$. Let $(i,k)$ be an element in $V\cap\mu_U^{-1}(\ed{ij})$.  Unless $j=k$ (and so $\edij\in V$), we have that
\begin{equation}(j,k)\in U-\mu_U^{-1}(\ed{ij}) < V\end{equation}
and so by the underset condition
\begin{equation} (i,k),(j,k)\in V\in\Gf \Rightarrow (i,j)\in V. \end{equation}
Therefore we must have $V=U$ by transitivity and so $U-\mu_U^{-1}(\ed{ij})$ is maximal.
So by Proposition~\ref{prop_proper} the diagonal complex $\Gf$ is proper.
\end{proof}

We may draw the diagonal complex when $n=3$ as follows:
\begin{equation}\xy
(0,0)*{}="O"; (-7,12)*{}="A"; ( 7,12)*{}="B"; (14, 0)*{}="C"; ( 7,-12)*{}="D"; (-7,-12)*{}="E"; (-14, 0)*{}="F";
( 0,12)*{}="AB"; (10.5, -6)*{}="CD"; (-10.5, -6)*{}="EF";
"A"; "B" **\dir{-}; "C"; "D" **\dir{-}; "E"; "F" **\dir{-};
"A"; "EF" **\crv{"O"};
"B"; "CD" **\crv{"O"};
"C"; "AB" **\crv{"O"};
"D"; "EF" **\crv{"O"};
"E"; "CD" **\crv{"O"};
"F"; "AB" **\crv{"O"};
\endxy\end{equation}

\noindent To assign an $\sbrac{n}$-labelling to $\Gf$ we give the pair $(i,j)$ the label $i\in\sbrac{n}$.
By the definition of $\gf$ the $U'$ in $\gf(U)$ are corollas with base $i$ and so are labelled by $i$.
Hence this labelling of $X_n$ is compatible with the diagonal complex structure of $\Gf$, in fact it is the universal labelling.

Let $\bdY=(Y_1\clc Y_n)$ be an $n$-tuple of pointed spaces.
\begin{theorem}\label{thm_YGFFR}
The fundamental group $\pi_1(\bdY\Gf)$ is isomorphic to the Fouxe-Rabinovitch group $\FR(\pi_1(Y_1)\ala\pi_1(Y_n))$.
\end{theorem}
\begin{proof}
Using Theorem~\ref{thm_Gdc} the relations of a diagonal complex product of groups all come from either the summand groups or from the $U\in\Gf$ of dimension 2.  In $\Gf$ these are
\begin{equation}\label{eq_dim2forests}
\xytree{k&j\\i\ar[u]\ar[ur]&}, \quad
\xytree{k&j\\l\ar[u]&i\ar[u]}\quad\text{ and }\quad
\xytree{j\\i\ar[u]\\k\ar[u]}.\end{equation}
Let $G_i = \pi_1(Y_i)$, then a presentation for $\bdG\Gf$ has generators
\begin{equation}g_{ij} \end{equation}
for each $(i,j)\in X_n$ and $g\in G_i$.  The relations consist of
\begin{equation} g_{ij}h_{ij} = (gh)_{ij} \end{equation}
coming from the generating groups and
\begin{equation} \sbrac{g_{ik},g_{ij}}, \quad
\sbrac{g_{lk},g_{ij}} \quad\text{ and }\quad
\sbrac{g_{ki}g_{kj},g_{ij}},\end{equation}
corresponding respectively to the forests in~\eqref{eq_dim2forests}.
Equating $g_{ij}$ with $\alpha_j^{g_i^{-1}}\in\FR(G_1\ala G_n)$ gives an equivalence between this presentation and the presentation in Proposition~\ref{prop_FRpres}.
\end{proof}
\begin{theorem}\label{thm_MYYGFiso}
The moduli space of cactus products $\MY$ is isomorphic to $\YGf$.
\end{theorem}
\begin{proof}
By Proposition~\ref{prop_moduliembedding}, the moduli space $\MY$ embeds into 
\begin{equation} Y_1^{n-1}\tlt Y_n^{n-1}.\end{equation}
We will now prove that this coincides with the embedding of $\YGf$ into $\bdY^{X_n}$, see Proposition~\ref{prop_dcproductcolimit}.
For each tree $t$, there is a planted forest $U$ which is isomorphic to $t$.  The spaces $\bdY^U$ and 
\begin{equation} t(\bdY):=\cbrac{y(T)\mid T\text{ is a cactus diagram over }t}\end{equation}
have the same image in $\bdY^X$.  Hence the moduli space is contained in $\YGf$.
Now for a planted forest $U\in\Gf$, there exists a maximal, connected planted forest $V$ containing $U$.
Each edge of $U$ is contained in $V$ and so each $U'\in\gf(U)$ is contained in some $V'\in\gf(V)$.
Therefore $\bdY^U$ is contained in $\bdY^V$ and since $V$ is a tree $t$, this shows that $\YGf$ is contained in the moduli space.
\end{proof}

\newcommand{\bGf}{\bG_\Forn}
\subsection{Bipartite planted forests}
The diagonal complex $\Gf$ defines a category $\bG_\Forn$, see Section~\ref{section_defcolimit}.
Recall from Section~\ref{section_partitions} that the set $\mathcal{X}_n$ of partial partitions of $X_n$ forms a poset $(\mathcal{X}_n,\leq_{pc})$ with the partial coarsening relation.
The category $\bGf$ is defined to be the subposet of $\mathcal{X}_n$ containing $\gf(\Gf)$ and closed under taking meets.
So to describe $\bGf$ we need to understand not only the forest posets but also their meets in $(\mathcal{X}_n,\leq_{pc})$.

A \emph{bipartite planted forest} on $\sbrac{n}$ is a planted forest $f$ with vertex set $\sbrac{n}\cup N$ such that
\begin{itemize}
\item all leaves and isolated vertices of $f$ (the vertices of valence $1$ and $0$ respectively), are in~$\sbrac{n}$,
\item each edge has one end in $\sbrac{n}$ and one end in $N$.
\end{itemize}
The second condition states that the underlying graph is bipartite, whilst the first says that each extremal vertex is in $\sbrac{n}$.
A planted forest is naturally directed by setting each edge to be directed away from the root of the tree in which it is contained.
As such each internal vertex $v$ has a unique \emph{parent} vertex $p(v)$ such that $\ed{p(v)v}$ is an edge.
Since each vertex in $N$ is internal this means that every such vertex has a unique parent, so $p$ restricts to a map $N\rightarrow \sbrac{n}$.

For each $v\in N$ define a subset $U_v\subset X_n$ by
\begin{equation}
U_v = \cbrac{(p(v), w)\mid\text{ there exists a directed path from $p(v)$ to $w$ in $f$}}.
\end{equation}
Then $\cbrac{U_v}_{v\in N}$ is a partial partition of $X_n$ associated to the bipartite planted forest~$f$.

Each planted forest $f$ defines a bipartite planted forest by setting $N=E$, the edge set of $f$.  The edges of the new forest are then given by $\ed{ve}$ and $\ed{ew}$ for each edge $e=\ed{vw}\in E$.
This is the barycentric subdivision of the forest.  The associated partial partition is equal to $\gf(f)$.

The partial partitions have an ordering by partial coarsening; $\cUi\leq_{pc}\cVj$ if for each $U_k\in\cUi$ there exists a subset of $\cVj$ whose union is $U_k$.
In this ordering the minimal elements above a partial partition $\cUi$ are given (a) by taking the union of two of the subsets, or (b) by removing one of the subsets.
There is an induced ordering of the bipartite planted forests.  Let $f$ be a bipartite planted forest and let 
\begin{equation}
\xytree{i &&&& j \\ &x\ar[ul]&&y\ar[ur]& \\ &&k\ar[ur]\ar[ul]&&}
\end{equation}
be a subtree, where $i,j,k\in\sbrac{n}$ and $x,y\in N$.
Now define a new forest $f_{xy}$ by identifying $x$ and $y$ to obtain $z$, the subtree now becomes
\begin{equation}\label{eq_Ytree}
\xytree{i && j \\ &z\ar[ur]\ar[ul]& \\ &k\ar[u]&}.
\end{equation}
Let us now consider the associated partial partitions.  If $w\in N$ is not $x$ or $y$ then $U_w$ is left unchanged by the operation of identifying $x$ and $y$.  This accounts for all subsets associated to $f_{xy}$ except $U_z$, this is equal to $U_x\cup U_y$.
We call the operation of identifying vertices $x$ and $y$ a \emph{horizontal folding}.
Now suppose that we instead have the following subtree
\begin{equation}
\xytree{i \\ x\ar[u] \\ j\ar[u] \\ y\ar[u] \\ k\ar[u]}.
\end{equation}
We again identify $x$ and $y$ to obtain $z$ and the subtree looks like the $Y$-shaped tree~\eqref{eq_Ytree}.
However this time the effect on partial partitions is to remove $U_x$.
We call this a \emph{vertical folding}.
For two bipartite planted forests $f$ and $f'$ we write $f' \leq f$ if $f'$ is obtained by a chain of foldings from $f$ to $f'$.  We then observe that the associated poset is the same poset structure as the one induced from the poset structure from the partial partitions.

\begin{proposition}
The set of partial partitions of $X_n$ associated to the bipartite planted forests is equal to the object set of the category $\bG_\Forn$.
\end{proposition}
\begin{proof}
To show this we need only show that every bipartite planted forest corresponds to the meet in $\mathcal{X}_n$ of partial partitions coming from forest posets.
So let $f$ be a bipartite planted forest with vertex set $\sbrac{n}\cup N$ and corresponding partial partition $\cbrac{U_v}_{v\in N}$.  Every vertex in $N$ of valence greater than two may be unfolded in a variety of ways.  Repeating this will give a forest in which each vertex in $N$ is bivalent and so is the barycentric subdivision of a planted forest with vertices in $\sbrac{n}$.  We claim that $\cbrac{U_v}$ is the meet of all the partial partitions corresponding to maximal unfoldings, this would complete the proof.

Let $\cVj$ be the meet of all the maximal unfoldings of $f$.  
So $\cVj$ is the greatest partial partition which is less than each unfolding, since $\cbrac{U_v}$ is less than each unfolding it must be less than (or equal to) $\cVj$.
Now consider the maximal $f_h$ given by horizontal unfoldings, the partial partition $\cWk$ of this has the same union as $\cbrac{U_v}$, the partial partition of $f$.
Since we have $\cbrac{U_v}\leq_{pc} \cVj\leq_{pc} \cWk$, the union of $\cVj$ must be the same as that of $\cbrac{U_v}$.

Now pick a maximal vertical unfolding $f_v$.
The partial partition $\cWk$ associated to $f_v$ contains $\cbrac{U_v}_{v\in N}$ as a subpartition.  
Since $\cbrac{U_v}\leq_{pc} \cVj\leq_{pc}\cWk$ this means that $\cVj$ also contains $\cbrac{U_v}$ as a subpartition.  However their unions are the same, so they must be equal.

Therefore $\cbrac{U_v}_{v\in N}$ is the meet and we are done.
\end{proof}
So we now have a combinatorial description of the category $\bGf$ in terms of bipartite planted forests and foldings.  To see how they give the diagonal complex product associated to $\Gf$ we need only describe the associated functor.
So let $\bdY=\cbrac{Y_1\clc Y_n}$ be an $n$-tuple of cocommutative coalgebras in a symmetric monoidal category $\cat{C}$.
We assign to a bipartite planted forest $f$ the product
\begin{equation}\label{eq_product}
\bdY_f:=\bigotimes_{x\in N} Y_{p(x)}
\end{equation}
Suppose that $f'$ is obtained by a horizontal folding identifying $x$ and $y$ in $N$ to give a vertex $z$.  Then $p(x)=p(y)=p(z)$ and so there is a map $\bdY_{f'}\rightarrow\bdY_f$ induced by the coalgebra map $Y_{p(z)}\rightarrow Y_{p(z)}\otimes Y_{p(z)}\cong Y_{p(x)}\otimes Y_{p(y)}$.

Now suppose that $f'$ is obtained by a vertical folding which identifies $x$ and $y$ to $z$ where $x<y$.  Then $p(z) = p(x)$ and the map $\bdY_{f'}\rightarrow \bdY_f$ is given by the unit map $Y_{p(z)}\otimes k \rightarrow Y_{p(x)}\otimes Y_{p(y)}$.
The diagonal complex product $\bdY\Gf$ is the colimit of this functor $F_\bdY:\bG_\Forn\rightarrow\cat{C}$.

\subsection{Based McCullough-Miller space}
In~\cite{McCullough1996} for each free product $G=G_1\ala G_n$ a space $MM_G$ on which $\smallcaps{OWh}(G):=\Wh(G) / \Inn(G)$ acts was described.  Their main result was that $MM_G$ was contractible and that the simplex stabilisers were of the form $\prod_{j=1}^k G_{i_j}$ for some $i_j$'s.
The equivariant spectral sequence of this action has been used in papers such as~\cite{Jensen2006} and~\cite{Berkove2010} to study the homology of $\smallcaps{OWh}(G)$ and hence of $\Wh(G)$.
The spaces $MM_G$ were defined using bipartite trees with an ordering by folding, along with a `marking' which consists of a basis of $G$.

In~\cite{Chen2005} a variant of McCullough-Miller space was studied.  In this variant the bipartite trees had a chosen basepoint $\ast$.
The space was denoted $L(G)$ and had an action of $\FR(G)$ rather than the outer version.
In was shown that the methods of~\cite{McCullough1996} could be applied to show that $L(G)$ is contractible.
Ofcourse based trees are equivalent to planted forests, one just removes the base vertex and the neighbouring vertices describe the planting.  So we may described $L(G)$ using the bipartite planted forests which we described above.

The following definition is a restatement of the definition of $L(G)$ from~\cite{Chen2005}.
\begin{definition}
An automorphism is \emph{carried by a bipartite planted forest} $f$ if it is in the subgroup generated by elements $\alpha^g_{U_x}:=\alpha^g_{i_1}\ldots\alpha^g_{i_k}$ where $x\in N$, $g\in G_{p(x)}$ and $U_x=\cbrac{(p(x),i_j)}$.
This subgroup is denoted $\bdG_f$ and is isomorphic to the term given by~\eqref{eq_product}.

A marked bipartite planted forest is a pair $([\alpha],f)$ where $[\alpha]$ is a left coset in $\FR(G)/\bdG_f$.
The marked bipartite planted forests form a poset where $([\alpha],f)\leq ([\beta],f')$ if $[\alpha]\subseteq[\beta]$ and $f$ is given by a chain of foldings of $f'$.
We denote this poset $M(G)$.
There is an action by $\FR(G)$ given by multiplication of the cosets on the left.
The space $L(G)$ is the geometric realisation of $M(G)$, the poset of marked bipartite planted forests.
\end{definition}

In the paper~\cite{Chen2005} it is shown that 
\begin{theorem}[(3.1~from~\cite{Chen2005})]\label{thm_contractibility}
The space $L(G)$ is contractible.
\end{theorem}
Note that although~\cite{Chen2005} only considers finite groups $G_i$, this is only applied from Section 4 onwards, so Theorem~3.1 is valid for all groups $G_i$.

Note also that in the definition above the poset $M(G)$ is precisely the poset of cosets in the family $\cbrac{\bdG_f}$ indexed by $\bG_\Forn$.
Therefore $L(G)$ is the coset complex $\CCGGf$ associated to the diagonal complex $\Gf$.
Therefore by Theorem~\ref{thm_contractibility} and Theorem~\ref{thm_cosetcomplex} we have our main theorem.

\setcounter{mainthm}{0}
\begin{mainthm}\label{thm_aspherical}
Let $\bdY=(Y_1\clc Y_n)$ be aspherical pointed spaces.  Then the space $\mathbf{Y}\Gf\cong \MY$ is also aspherical.
\end{mainthm}

\section{The (co)homology of automorphism groups}\label{section_HAutFP}
In order to prove Theorem~\ref{thm_homology} it remains to read off the homology of $\MY$ from its presentation as a diagonal complex product $\bdY\Gf$.
Then to finish we study the action of $\bAut(\bdG)\rtimes\symG$ on the chain complex of $\FR(G)$ to give Theorem~\ref{thm_aut}; a calculation of the integral homology of the symmetric automorphism group of a free product.
\subsection{The Fouxe-Rabinovitch groups}
We start with the homology of the Fouxe-Rabinovitch groups, which by Theorems~\ref{thm_YGFFR} and~\ref{thm_aspherical} is given by $H_\ast(\YGf,R)$.
Let $U$ be a planted forest.  The monomial, see~\eqref{eq_monomial}, attached to $U$ is 
\begin{equation}\label{eq_forestmonomial}
x_1^{\out{1}}\ldots x_n^{\out{n}},\end{equation}
where $\out{i}$ is the number of outgoing edges of $U$ from $i$.
The following is a restatement of Theorem~5.3.4 of~\cite{Stanley2001}.
\begin{proposition}
The Hilbert-Poincar\'e series of $\Gf$ is
\begin{equation}h_{\Forn}(x_1\clc x_n) = (1+x_1+\ldots +x_n)^{n-1}.\end{equation}
\end{proposition}
We refer to~\cite{Stanley2001} for the full proof, but in the interests of self-containment we sketch part of the proof below.
\begin{proof}[(sketch)]
A proof of this proposition involves the theory of \emph{Pr\"ufer codes}.
The set of planted forests with vertex set $\sbrac{n}$ is shown to be in bijection with the set of words of length $n-1$ in the letters $\cbrac{x_0,x_1\clc x_n}$.
To give the word of a planted forest, attach a new vertex $0$ onto the forest by adding the edge $\ed{0i}$ for each of the roots $i$.
A \emph{leaf} is a vertex which is attached to only one other vertex.
A word $w=s_1\ldots s_{n-1}$ is produced for each tree $t$ as follows:
\begin{enumerate}
\item Let $t_1 = t$.
\item Suppose that $t_i$ is defined.  Take the leaf $v$ of maximal value in $t_i$ and define $s_{i}$ to be the value of its unique adjacent vertex.
\item If $i=n-1$ then stop.
\item Now define $t_{i+1}$ to be the tree created by removing $v$ and its unique adjoining edge from $t_i$.
\item Go back to step (2).
\end{enumerate}
The following diagram illustrates an example; the forest on the left is turned into the word $x_2x_1x_5x_0x_2$.
\begin{equation}\xytree{
& 4 & & *+[o][F-]{4} &&&&&&\\
*+[o][F-]{6}& 1 \ar[u] & 3 & 1\ar[u] & 3 & 1 & *+[o][F-]{3} & 1 && *+[o][F-]{1} \\
& 2\ar[ul]\ar[u] & 5\ar[u] & 2\ar[u] & 5\ar[u] & 2\ar[u] & 5\ar[u] & 2\ar[u] & *+[o][F-]{5} & 2\ar[u] \\
& 0\ar@{.>}[u]\ar@{.>}[ur] & & 0\ar@{.>}[u]\ar@{.>}[ur] & & 0\ar@{.>}[u]\ar@{.>}[ur] & & 0\ar@{.>}[u]\ar@{.>}[ur] & & 0\ar@{.>}[u] \\ 
& 2 & & 21 & & 215 & & 2150 & & *+[F-]{21502}
}\end{equation}
This algorithm may be reversed in order to build a tree from a word and this produces the bijection.  The element
\begin{equation}\label{eq_monoidsum} (x_0+x_1+\ldots+x_n)^{n-1}\end{equation}
in the free commutative ring on $\cbrac{x_0\clc x_n}$ is the sum over all words of length $n-1$.
The word associated with a planted forest may be turned into a monomial by viewing it in the free commutative monoid on $\cbrac{x_1\clc x_n}$ and setting $x_0=1$.
Since the number of $x_i$'s in a word is the number of outgoing edges of $i$, this monomial is the same as the monomial defined in~\eqref{eq_forestmonomial}.
Hence setting $x_0=1$ and viewing~\eqref{eq_monoidsum} in the free commutative ring gives the Hilbert-Poincar\'e series of the planted forests.
\end{proof}
The series for the planted forests is identical to the Hilbert-Poincar\'e series of the diagonal complex describing $(G_1\ala G_n)^{\times n-1}$ and so by Corollary~\ref{cor_hilb}:
\begin{mainthm}\label{thm_homology}
Let $\bdG$ be an $n$-tuple of groups and $G$ be the $n$-fold free product of these groups.
Then 
\begin{equation}H_\ast(\FR(G),R) \cong H_\ast(G^{n-1},R),\end{equation}
where $G^{n-1}$ is the $(n-1)$-fold direct product of $G$.
This also holds for the cohomology.
\end{mainthm}
The Euler characteristic of $\Wh(\ZZ^{\ast n})\cong \FR(\ZZ^{\ast n})$, was computed in~\cite{McCammond2004}, and the homology was calculated in~\cite{Jensen2006}.  
Recalling that the Hilbert-Poincar\'e series of $\ZZ$ is $1+t$ we may reobtain these results.
\begin{corollary}
The Hilbert-Poincar\'e series of $H_\ast\bbrac{\Wh(F_n),\ZZ}$ is
\begin{equation} (1+tn)^{n-1}\end{equation}
and so the Euler characteristic is
\begin{equation} (1-n)^{n-1}.\end{equation}
\end{corollary}
\newcommand{\Zp}{\ZZ/(p)}
We may take this further by calculating the homology of $\Wh\bbrac{\Zp^{\ast n}}$.
The question of the cohomology of this group was posed by Jensen in~\cite{Mislin2008}.
\begin{corollary}
The Hilbert-Poincar\'e series of $H_\ast\bbrac{\Wh(\Zp^{\ast n}),\ZZ}$ is
\begin{equation}\label{eq_hilbzp}
1 + y\dfrac{1}{1+t}\sbrac{\brac{1+\dfrac{nt}{1-t}}^{n-1}\! - 1},
\end{equation}
where the $y$ coefficient encodes the number of $\Zp$-summands and the constant term encodes the number of $\ZZ$-summands.
\end{corollary}
\begin{proof}
The homology of $\Zp$ is given by
\begin{equation}H_i(\Zp,\ZZ) = \begin{cases} \ZZ & \text{ if $i=0$,}\\
\Zp & \text{ if $i=2k-1$ for $k\geq 1$ and}\\
0 & \text{ if $i=2k$ for $k\geq 1$.}\\
\end{cases}\end{equation}
As a Hilbert-Poincar\'e series this is 
\begin{equation} 1 + yt + yt^3 + yt^5 + \ldots = 1+\dfrac{ty}{1-t^2}. \end{equation}
Putting this into the Hilbert-Poincar\'e series for $\Gf$ gives
\begin{equation}\brac{1+\dfrac{nyt}{1-t^2}}^{n-1} = 1+\sum_{k=1}^{n-1} \binom{n-1}{k} \brac{\dfrac{nyt}{1-t^2}}^k. \end{equation}
Using the identity $y^2 = y(1+t)$ from~\eqref{eq_tormult} to give $y^k = y(1+t)^{k-1}$ and simplifying we get the desired series~\eqref{eq_hilbzp}.
\end{proof}
\subsection{The pure automorphism groups}
Recall that the pure automorphism group, $\PAut(G)$ consists of those automorphisms that take an element $g\in G_i$ and give a conjugate $(g')^h$ of an element $g'\in G_i$ inside $G=G_1\ala G_n$ .
It is a semidirect product of $\FR(G)$ and $\bAut(\bdG) \cong \prod_i\Aut(G_i)$.
The action of $\bAut(\bdG)$ on $\FR(G)$ may be described as follows.  Let $\phi\in\Aut(G_i)$, then 
\begin{equation}(\alpha_k^{g_j})^\phi = \begin{cases}\alpha_k^{g_j} & \text{ if $j\neq i$ and }\\
\alpha_k^{\phi(g_j)} & \text{ if $j=i$.}\end{cases}\end{equation}
The group $\bAut(\bdG)$ has an action on the free product $G$ and hence on $G^{n-1}$ by the diagonal action, so we may form the group
\begin{equation}G^{n-1}\rtimes \bAut(\bdG).\end{equation}
\begin{theorem}\label{thm_pureauts}
The homology of $\PAut(G)$ decomposes as
\begin{equation}
H_\ast\bigl(\bAut(\bdG),R\bigr)\oplus\bigoplus_{F\in\Gf} H_\ast\bbrac{\bAut(\bdG),\rC(F)},\end{equation}
where
\begin{equation}
\rC(F) = \bigotimes_{\edij\in E(F)} \rC(Y_{i})
\end{equation}
and $Y_i$ is a classifying space for $G_i$ for each $i=1\clc n$.

Therefore 
\begin{equation} H_\ast\bbrac{\PAut(G),R}\cong H_\ast\bbrac{G^{n-1}\rtimes \bAut(\bdG),R}. \end{equation}
\end{theorem}
\begin{proof}
The diagonal complex product $\bdY\Gf$ is a classifying space for $\FR(G)$ and has an action of $\bAut(\bdG)$.
The homotopy quotient is then a classifying space for the pure automorphism group, $\PAut(G)=\FR(G)\rtimes\bAut(\bdG)$.

Let $\Gf$ have the labelling by $\sbrac{n}$ given by $l(i,j) = i$.
For each $i\in\sbrac{n}$ let $H_i=\Aut(G_i)$ be the automorphism group of $G$.
The direct product, $H$ of the $H_i$ acts on $\bdY\Gf$ via~\eqref{eq_Haction}.
This direct product is isomorphic to $\bAut(\bdG)$ and has the same action on $\bdY\Gf$.
Letting $\sym$ be the trivial group, then we may apply Theorem~\ref{thm_homotopyquotient} to calculate the homology of $\PAut(G)$, which gives the decomposition.

For the second part observe that for any planted forest $F\in\Gf$ with the same monomial as a term $U$ in the homology of $G^{n-1}$, the automorphism groups always act diagonally and so $\rC(F)\cong\rC(U)$ as $H$-modules.
Since the Hilbert-Poincar\'e polynomials are identical the corresponding sums of $H$-modules are isomorphic.
Hence applying Theorem~\ref{thm_homotopyquotient} for both groups $\PAut(G)$ and $G^{n-1}\rtimes\bAut(\bdG)$ we get the same homology.
\end{proof}

\subsection{The symmetric automorphism groups}\label{section_fullautos}
Let $n_1\clc n_k$ be the sizes of the isomorphism classes of $\bdG=(G_1\clc G_n)$.
Then write $\symG$ for the group
\begin{equation} \sym_{n_1}\tlt\sym_{n_k} \leq \sym_n \end{equation}
of symmetries of $\bdG$.
The symmetric automorphism group $\SAut(G)$ is the semidirect product of the pure automorphism group $\PAut(G)$ and $\symG$.

Remember that the diagonal complex $\Gf$ has a universal labelling $X_n\rightarrow\sbrac{n}$.  
So a map $L:\sbrac{n}\rightarrow Z=\sbrac{k}$ induces a labelling $l:X_n\rightarrow Z$.
Let $n_i$ be the size of the set $L^{-1}(i)$ for $i\in Z$.
We will call a planted forest $Z$-coloured if there is a vertex colouring by $Z$ and if for each $i\in Z$ there are $n_i$ vertices coloured $i$.
\begin{lemma}\label{lemma_symorbits}
The $\cbrac{1}$-labelled diagonal complex $\Gf$ carries an action of $\sym_n$.
The orbits are given by the isomorphism types of unlabelled planted forests.

The $Z$-labelled diagonal complex $\Gf$ carries an action of $\symG=\sym_{n_1}\tlt \sym_{n_k}$.
The orbits are given by the isomorphism types of $Z$-coloured planted forests.
\end{lemma}
\begin{proof}
The symmetric group $\sym_n$ has the permutation action on $\sbrac{n}$.  This gives an action on $X_n=\cbrac{(i,j)\mid i\neq j}$.
The induced action on $\Pf X_n$ restricts to $\Gf$:  
an element $\sigma\in\sym_n$ takes the edge $\edij$ of $U\in\Gf$ to $\ed{\sigma(i)\sigma(j)}$ of $\sigma(U)$.

The action of $\symG$ on $\sbrac{n}$ fixes the morphism $L:\sbrac{n}\rightarrow Z$ and hence the action on $\Gf$ satisfies
\begin{equation} l(\sigma(U)) = l(U), \end{equation}
for all $\sigma\in\symG$ and $U\in \Gf$.
\end{proof}
\begin{example}
Let $L:\cbrac{1,2,3}\rightarrow\cbrac{\circ,\dt}$ be defined by $l(1)=l(2)=\circ$ and $l(3)=\dt$.
Then the orbits are enumerated by the $Z$-coloured planted forests as follows:
\begin{equation}\smtree{\dt\\\circ\ar[u]\\\circ\ar[u]}\quad\smtree{\circ\\\dt\ar[u]\\\circ\ar[u]}\quad\smtree{\circ\\\circ\ar[u]\\\dt\ar[u]}\quad
\smtree{\circ&\circ\\\dt\ar[u]\ar[ur]&}\smtree{\dt&\circ\\\circ\ar[u]\ar[ur]&}
\smtree{\circ&\\\dt\ar[u]&\ar@{}[l]|-*{\circ}}\smtree{\dt&\\\circ\ar[u]&\ar@{}[l]|-*{\circ}}\smtree{\circ&\\\circ\ar[u]&.\ar@{}[l]|-*{\dt}}
\end{equation}
\end{example}
\begin{mainthm}\label{thm_aut}
Let $G=G_1\ala G_n$ where $G_i$ is neither freely decomposable nor infinite cyclic.  Then the homology of the automorphism group of $G$ is given by
\begin{equation}
H_\ast\bigl(\Aut(G),R\bigr) \cong H_\ast\bigl(\bAut(\bdG)\!\rtimes\!\symG,R\bigr)\oplus\!\!\!
\bigoplus_{f\in\smallcaps{Forests}_Z}\!\!\! H_\ast\brac{\bAut(\bdG)\!\rtimes\!\Aut(f), \rC(f)},
\end{equation}
where $\FZ$ is the set of $Z$-coloured planted forests and $\rC(f)$ is given by
\begin{equation} \bigotimes_{i\in\sbrac{n}} \rC(Y_{l(i)},R)^{\otimes \out{i}},\end{equation}
where $Y_i$ is a classifying space for $G_i$ for each $i=1\clc n$.
\end{mainthm}
\begin{proof}
There are two labellings of $\Gf$, the labelling by $Z$ coming from the isomorphism type of the $G_i$ and the universal labelling by $\sbrac{n}$.
The group $\symG$ acts on $\Gf$ preserving the first labelling.
Thus we may apply Theorem~\ref{thm_homotopyquotient} to compute the homology of the homotopy quotient of the action of $\bAut(\bdG)\rtimes\symG$ on $\bdY\Gf$.
This homotopy quotient is a classifying space for $\SAut(G)$.

By Lemma~\ref{lemma_symorbits} the orbits of $\Gf$ are given by $\smallcaps{Forests}_Z$ and the stabiliser of a $Z$-coloured forest $f$ is $\Aut(f)$.
\end{proof}
\begin{remark}[cohomological version of Theorem~\ref{thm_aut}]
Precisely the same arguments may be used to prove that given the hypotheses of Theorem~\ref{thm_aut} the cohomology of $\Aut(G)$ is given by
\begin{equation}
H^\ast\bigl(\Aut(G),R\bigr) \cong H^\ast\bigl(\bAut(\bdG)\!\rtimes\!\symG,R\bigr)\oplus\!\!\!\!
\bigoplus_{f\in\smallcaps{Forests}_Z}\!\!\!\! H^\ast\brac{\bAut(\bdG)\!\rtimes\!\Aut(f), \rcC(f)},
\end{equation}
where $\rcC(f)$ is given by
\begin{equation} \bigotimes_{i\in\sbrac{n}} \rcC(Y_{l(i)})^{\otimes \out{i}}.\end{equation}
\end{remark}
\begin{remark}
Let $G = G_1^{\ast n_1}\ala G_k^{\ast n_k}$ with each $G_i$ neither freely decomposable nor $\ZZ$, then $\Aut(G)\cong\PAut(G)\rtimes\symG$, so we have that the Euler characteristic of $\Aut(G)$ is
\begin{equation}\dfrac{1}{(n_1)!\ldots (n_k)!} \bbrac{1+n_1\chi(G_1)+\ldots+n_k\chi(G_k)}^{n-1}\prod_{i=1}^k \chi\bbrac{\Aut(G_i)}^{-n_i},\end{equation}
which is the same as that of $(G^{n-1}\rtimes \bAut(\bdG) )\rtimes\symG$.
However their integral homologies are not necessarily the same due to the presence of the groups $\Aut(f)$.
Hence the pattern 
\begin{align}
H_\ast(\FR(G),R) &\cong H_\ast(G^{n-1},R), \label{eq_isom1}\\
H_\ast(\Wh(G),R) &\cong H_\ast(G^{n-1}\rtimes \bInn(\bdG),R), \label{eq_isom2}\\
H_\ast(\PAut(G),R) &\cong H_\ast(G^{n-1}\rtimes\bAut(\bdG),R)\label{eq_isom3}
\end{align}
is broken when non-pure symmetric automorphisms are present.
\end{remark}
\begin{remark}\label{rem_discrepancy}
There is a discrepancy between the three isomorphisms we obtain in~\eqref{eq_isom1},~\eqref{eq_isom2} and~\eqref{eq_isom3} and the results of~\cite{Jensen2007euler} and~\cite{Berkove2010}.
Calculating the Euler characteristic we find that 
\begin{equation}\chi(\FR(G))=\chi(G^{n-1}),\end{equation}
whereas in~\cite{Jensen2007euler} it is calculated to be 
\begin{equation}\chi(G^{n-1}).\prod_i\chi(\Inn(G_i))^{-1}{}_.\end{equation}
Across all of the results there is a factor of $\chi(\bInn(\bdG))$ difference.
There is an analogous difference with the results of~\cite{Berkove2010}.
We attribute this to a misquoting of Proposition 5.1 of~\cite{McCullough1996} in each of the papers.
\end{remark}
\begin{example}
Let $G_i=\ZZ$ and consider the group $\SFR(F_n) \cong \FR(F_n)\rtimes \sym_n$.  Using Theorem~\ref{thm_aut} we may calculate the homology as
\begin{align}
 H_\ast\bbrac{\SFR(F_n),R} &= H_\ast(\sym_n,R)\oplus \bigoplus_{f\in\smallcaps{Forests}} H_\ast\bbrac{\Aut(f),\rC(f)}\\
 &\cong H_\ast(\sym_n,R)\oplus \bigoplus_{f\in\smallcaps{Forests}} t^{|E(f)|} \,H_\ast\bbrac{\Aut(f), R_f},
\end{align}
where $t$ is of degree one and $R_f$ is the one-dimensional `determinant' $\Aut(f)$-module: a factor of $-1$ is introduced every time two edges of $f$ are swapped.
\end{example}
\bibliographystyle{plain}
\bibliography{library}

\begin{thebibliography}{10}

\bibitem{Abels1993}
H.~Abels and S.~Holz.
\newblock Higher generation by subgroups.
\newblock {\em Journal of Algebra}, 160(2):310--341, 1993.

\bibitem{Baez2007}
J.C. Baez, D.K. Wise, and A.S. Crans.
\newblock {Exotic statistics for strings in 4d BF theory}.
\newblock {\em Advances in Theoretical and Mathematical Physics}, 11(5):707,
  2007.

\bibitem{Berkove2010}
E.~Berkove and J.~Meier.
\newblock {The cohomology groups of the outer Whitehead automorphism group of a
  free product}.
\newblock {\em Forum Mathematicum}, 22(2):379--395, 2010.

\bibitem{Brendle2010}
T.~Brendle and A.~Hatcher.
\newblock {Configuration spaces of rings and wickets}.
\newblock {\em Commentarii Mathematici Helvetici}, 2010.

\bibitem{Brownstein1993}
A.~Brownstein and R.~Lee.
\newblock {Cohomology of the group of motions of n strings in 3-space}.
\newblock In {\em Mapping class groups and moduli spaces of Riemann surfaces
  (G{\"o}ttingen, 1991/Seattle, WA, 1991)}, volume 150, pages 51--61. American
  Mathematical Society, 1993.

\bibitem{Charney2007b}
R.~Charney.
\newblock {An introduction to right-angled Artin groups}.
\newblock {\em Geometriae Dedicata}, 125(1):141--158, March 2007.

\bibitem{Chen2005}
Yuqing Chen, Henry~H. Glover, and Craig~A. Jensen.
\newblock Proper actions of automorphism groups of free products of finite
  groups.
\newblock {\em Internat. J. Algebra Comput.}, 15(2):255--272, 2005.

\bibitem{Collins1989}
D.J. Collins.
\newblock {Cohomological dimension and symmetric automorphisms of a free
  group}.
\newblock {\em Commentarii Mathematici Helvetici}, 64(1):44--61, 1989.

\bibitem{Culler1986}
M.~Culler and K.~Vogtmann.
\newblock {Moduli of graphs and automorphisms of free groups}.
\newblock {\em Inventiones mathematicae}, 84(1):91--119, 1986.

\bibitem{Fouxe1940}
D.I. Fouxe-Rabinovitch.
\newblock {\"Uber die Automorphismengruppen der freien Produkte. I}.
\newblock {\em Mat. Sbornik}, 8(1):264--276, 1940.

\bibitem{Gilbert1987}
N.D. Gilbert.
\newblock {Presentations of the automorphism group of a free product}.
\newblock {\em Proceedings of the London Mathematical Society}, 3(1):115, 1987.

\bibitem{Guirardel2007}
V.~Guirardel and G.~Levitt.
\newblock {The outer space of a free product}.
\newblock {\em Proceedings of the London Mathematical Society}, 94(3):695,
  2007.

\bibitem{Hatcher2002}
A.~Hatcher.
\newblock {\em {Algebraic topology}}.
\newblock Cambridge University Press, 2002.

\bibitem{Jensen2006}
C.~Jensen, J.~McCammond, and J.~Meier.
\newblock {The integral cohomology of the group of loops}.
\newblock {\em Geometry \& Topology}, 10:759--784, 2006.

\bibitem{Jensen2007euler}
C.~Jensen, J.~McCammond, and J.~Meier.
\newblock {The Euler characteristic of the Whitehead automorphism group of a
  free product}.
\newblock {\em Transactions of the American Mathematical Society}, 359(6):2577,
  2007.

\bibitem{Leary1997}
I.J. Leary.
\newblock {On the integral cohomology of wreath products}.
\newblock {\em Journal of Algebra}, 198(1):184--239, 1997.

\bibitem{McCammond2004}
J.~McCammond and J.~Meier.
\newblock {The hypertree poset and the l 2-Betti numbers of the motion group of
  the trivial link}.
\newblock {\em Mathematische Annalen}, 328(4):633--652, 2004.

\bibitem{McCullough1996}
D.~McCullough and A.~Miller.
\newblock {\em {Symmetric automorphisms of free products}}.
\newblock American Mathematical Society, 1996.

\bibitem{Mislin2008}
G.~Mislin and I.~Chatterji.
\newblock {\em {Guido's book of conjectures: a gift to Guido Mislin on the
  occasion of his retirement from ETHZ, June 2006}}.
\newblock Enseignement math{\'e}matique, 2008.

\bibitem{Stanley2001}
R.P. Stanley.
\newblock {\em {Enumerative combinatorics: Volume 2}}.
\newblock Cambridge University Press, 2001.

\bibitem{Vogtmann2002}
K.~Vogtmann.
\newblock {Automorphisms of free groups and outer space}.
\newblock {\em Geometriae Dedicata}, 94(1):1--31, 2002.

\end{thebibliography}

\end{document}